\documentclass[oneside]{amsart}
\usepackage{geometry}
 \usepackage{graphicx}
 \usepackage{amssymb}
 \usepackage[mathscr]{euscript}
\usepackage{amsmath}
\usepackage{amsfonts}
\usepackage{mathtools}
\usepackage{amsthm}
\usepackage{hyperref}
\usepackage{enumerate}
\usepackage{indentfirst}
\usepackage{tikz}
\usepackage{comment}
\usetikzlibrary{arrows,automata,positioning}

\allowdisplaybreaks

\DeclarePairedDelimiter{\ceil}{\lceil}{\rceil}

\usepackage[compress, capitalize]{cleveref}
\newtheorem{thm}{Theorem}
\crefname{thm}{Theorem}{Theorems}
\newtheorem{prop}[thm]{Proposition}
\crefname{prop}{Proposition}{Propositions}
\newtheorem{lem}[thm]{Lemma}
\crefname{lem}{Lemma}{Lemmas}
\newtheorem{fact}[thm]{Fact}
\crefname{fact}{Fact}{Facts}

\newtheorem{cor}[thm]{Corollary}
\crefname{cor}{Corollary}{Corollaries}

\newtheorem*{thm*}{Theorem}
\newtheorem*{thmA}{Theorem A}
\newtheorem*{thmB}{Theorem B}

\numberwithin{thm}{section}

\theoremstyle{definition}
\newtheorem{defn}[thm]{Definition}
\newtheorem{ex}[thm]{Example}

\newcommand{\N}{\mathbb{N}}
\newcommand{\R}{\mathbb{R}}

\newcommand{\Z}{\mathbb{Z}}
\newcommand{\Q}{\mathbb{Q}}

\newcommand{\bd}{\operatorname{bd}}
\newcommand{\norm}[1]{\left\lVert#1\right\rVert}

\newcommand{\floor}[1]{\left\lfloor #1\right\rfloor}
\newcommand{\abs}[1]{\left|#1\right|}
\newcommand{\set}[1]{\left\{#1\right\}}
\newcommand{\inter}[1]{\left[#1\right]}
\newcommand{\lointer}[1]{\left(#1\right]}
\newcommand{\rointer}[1]{\left[#1\right)}

\newcommand{\mc}[1]{\mathcal{#1}}

\newcommand{\eps}{\varepsilon}
\renewcommand{\phi}{\varphi}
\renewcommand{\theta}{\vartheta}

\newcommand{\xto}{\xrightarrow}

\author[P. Hieronymi]{Philipp Hieronymi}
\address{Mathematical Institute\\ University of Bonn\\
Endenicher Allee 60\\ 53115 Bonn\\ Germany}
\email{hieronymi@math.uni-bonn.de}

\author[S. Manthe]{Sven Manthe}
\address{Mathematical Institute\\ University of Bonn\\
Endenicher Allee 60\\ 53115 Bonn\\ Germany}
\email{sven.manthe@uni-bonn.de}

\author[C. Schulz]{Chris Schulz}
\address{Department of Pure Mathematics\\
200 University Avenue West\\
Waterloo, Ontario\\
N2L 3G1\\
Canada
}
\email{chris.schulz@uwaterloo.ca}
\title{A Cobham theorem for scalar multiplication}
\begin{document}

\begin{abstract}
Let $\alpha,\beta \in \R_{>0}$ be such that $\alpha,\beta$ are quadratic and 
$\Q(\alpha)\neq \Q(\beta)$. Then every subset of $\R^n$ definable in both
$(\R,{<},+,\Z,x\mapsto \alpha x)$ and $(\R,{<},+,\Z,x\mapsto \beta x)$ is already definable in $(\R,{<},+,\Z)$. As a consequence we generalize Cobham-Sem\"enov theorems for sets of real numbers to $\beta$-numeration systems, where $\beta$ is a quadratic irrational.
\end{abstract}

\maketitle

\section{Introduction}

\noindent This paper is part of a larger enterprise to study expansions of the structure $(\R,{<},+, \Z)$ by fragments of multiplication.  As an easy consequence of Büchi’s theorem on the decidability of the monadic second-order theory of one successor \cite{Buechi}, the first-order theory of $(\R,{<},+, \Z)$ is known to be decidable. Arguably due to Skolem \cite{skolem}, but later independently rediscovered by Weispfenning \cite{Weispf} and Miller \cite{miller-ivp}, it has quantifier-elimination in the larger signature expanded by function symbols for the floor function and $x\mapsto qx$ for each $q\in \Q$. These results haven been implemented and used in algorithm verification
of properties of reactive and hybrid systems, see for example \cite{LIRA, BJW, FQSW,VIRAS}. Therefore it is only natural to consider more expressive expansions of this structure, in particular by fragments of multiplication. Gödel’s famous first incompleteness theorem obviously implies that  expanding $(\R,<,+, \Z)$ by a symbol for multiplication on $\R$ results in an structure with an undecidable theory.
However, even substantially smaller fragments of multiplication yield undecidability. Let $\alpha$ be a real number and let $\mathcal{R}_{\alpha}$ denote the structure $(\R,{<},+,\Z,x\mapsto \alpha x)$. By \cite[Theorem A]{H-multiplication} the first-order theory of $\mathcal{R}_{\alpha}$ is only decidable when $\alpha$ is a quadratic number, that is, a solution to a quadratic polynomial equation with integer coefficients. Thus studying $\mathcal{R}_{\alpha}$ and its reducts and fragments has become an active area of research, in particular when $\alpha$ is quadratic (see for example \cite{hnp,decidsturmian,KhZ}). Here, we answer a question raised in \cite[Question 6]{H-Twosubgroups}. Throughout this paper we follow the convention common in theoretical computer science that definability means definability without parameters.
\begin{thmA}
Let $\alpha,\beta \in \R_{>0}$ be such that $\alpha,\beta$ are quadratic and $\Q(\alpha)\neq \Q(\beta)$, and let $X\subseteq \R^n$ be definable in both $\mathcal{R}_{\alpha}$ and $\mathcal{R}_{\beta}$. Then $X$ is definable in $(\R,{<},+,\Z)$.
\end{thmA}

\noindent By Hieronymi and Tychoniviech \cite{HTycho} the structure $(\R,{<},+,\Z,x\mapsto \alpha x,x \mapsto \beta x)$ defines with parameters multiplication on $\R$ whenever $1,\alpha,\beta$ are $\Q$-linearly independent, and hence defines with parameters every projective set in the sense of descriptive set theory.  It follows by quantifying over parameters needed to define multiplication that the first-order theory of the structure
is undecidable and that every arithmetical subset of $\N^n$ is definable without parameters. Thus the structure generated by sets definable in \emph{either} $\mathcal{R}_{\alpha}$ \emph{or} $\mathcal{R}_{\beta}$ is as logically complicated,  as can be, while by our Theorem A the structure generated by all sets definable in \emph{both} $\mathcal{R}_{\alpha}$ \emph{and} $\mathcal{R}_{\beta}$ is as simple as it could be.\newline

\noindent Theorem A can be seen as an analogue of Cobham's theorem for scalar multiplication. Indeed, as we explain now, the connection is much stronger and more direct than it might appear. The famous theorems of Cobham \cite{Cobham} and Sem\"enov \cite{Semenov} state that for multiplicatively independent $k, \ell \in \N_{>1}$  every subset of $\N^n$ that is both $k$- and $\ell$-recognizable is already definable in $(\N, +)$. This result has been generalized in many direction, both in terms of numeration systems and in terms of the underlying domain $\N$. See Durand and Rigo \cite{durandrigo} for a survey.
A \textbf{Pisot number} is a real algebraic integer greater than 1 all whose Galois conjugates are less than 1 in absolute value. Bès \cite{Bes-CS} shows that for multiplicatively independent Pisot numbers $\alpha,\beta$, and for two linear numeration system $U$, $U'$ whose characteristic polynomials are the minimal polynomials of $\alpha$ and $\beta$ respectively, a subset of $\N^n$ that is both $U$- and $U'$-recognizable is definable in $(\N,+)$. Boigelot et al \cite{BBB10, BBL09} extend the Cobham-Sem\"enov theorem to subsets of $\R$ showing that for multiplicatively independent $k, \ell \in \N_{>1}$ every subset of $\R^n$ that is both weakly $k$- and weakly $\ell$-recognizable is already definable in $(\R, <, +,\Z)$. In this setting over $\R$, it is natural to consider so-called $\beta$-numeration systems introduced by R\'{e}nyi \cite{Renyi}, in which the usual integer base $k$ is replaced by a real number $\beta$ larger than $1$. For details, see \cite[Chapter 7]{Lothaire}. Charlier et al \cite{CLR} introduce the corresponding notion of $\beta$-recognizability. Here we prove an extension of the results from \cite{BBB10, BBL09} to such numeration systems.  
\begin{thmB}
Let $\alpha,\beta  \in \R_{>1}$ be multiplicatively independent irrational Pisot numbers such that $\Q(\alpha)\cap\Q(\beta)=\Q$, and let $X\subseteq[0,1]^d$ be both $\alpha$- and $\beta$-recognizable. Then $X$ is definable in $(\R,{<},+,\Z)$.
\end{thmB}

\noindent Theorem B at least partially answers a question of Rigo \cite[p. 48]{rigo}.
When $\alpha,\beta$ are quadratic, we have that $\Q(\alpha)\cap\Q(\beta)=\Q$ if and only if $\Q(\alpha)\neq \Q(\beta)$. Moreover, in this situation $\alpha$ and $\beta$ are multiplicatively independent (see \cref{quadirrmulind} for a proof).
Thus in order to deduce Theorem A from Theorem B, we need to reduce to the bounded case and establish the equivalence of $\alpha$-recognizability and definability in $\mathcal{R}_{\alpha}$. Various results showing the equivalence between recognizability and definability exist and often go back to B\"uchi's original work. For example, see Bruy\`ere et al \cite{BHMV} and Boigelot et al \cite{BRW} for the equivalence between $k$-recognizable and $k$-definable for $k\in \N$. The equivalence between $\beta$-recognizability and $\beta$-definability for Pisot numbers is established in \cite[Theorem 16 \& 22]{CLR}, but no argument is made there that their notion of $\beta$-definability corresponds to definability in $\mathcal{R}_{\beta}$. Here we provide this extra argument in \cref{alphgammreg}, showing how $\alpha$-recognizability corresponds to definability in $\mathcal{R}_{\alpha}$ for quadratic irrational $\alpha$. Indeed, we also show that these notions of recognizability and definability coincide with notions of recognizability using Ostrowski numeration systems instead of $\beta$-numeration systems. This should be of interest in its own right, as it explains the connection between the a priori unrelated work in \cite{CLR} and \cite{H-Twosubgroups}.\newline 

\noindent There is a subtlety regarding recognizability that we need to address in this setting. In \cite{BRW} a positive real number in $k$-ary is encoded as an infinite word over $\{0,\dots,k-1\}\cup \{\star\}$, where $\star$ serves as the radix point\footnote{For ease of exposition, we don't worry here about the encoding of negative numbers.}. Loosely speaking, a positive real number $a$ such that 
\[
a = b_{m} k^m + \dots + b_1 k + b_0 + b_{-1} k^{-1} + \dots,
\]
is encoded as the infinite word $b_m\cdots b_1b_0\star b_{-1}\cdots$. Thus in this setup the integer part of $a$ and the fractional part of $a$ are read \emph{sequentially}. However, there are also encodings that read the integer part and the fractional part in \emph{parallel}. Such an encoding is used for example in Chaudhuri et al \cite{CSV}. Here, the above real number $a$ corresponds to the infinite word
\[
(b_0,b_{-1})(b_1,b_{-2})\cdots (b_{m},b_{-(m+1)})(0,b_{-(m+2)})\cdots
\]
over the alphabet $\{0,\dots,k-1\}^2$. The second encoding is strictly more expressive, as the function mapping $k^i$ to $k^{-i}$ is recognizable using this parallel encoding. In previous work this difference has not been addressed, but here it becomes necessary as definability in $\mathcal{R}_{\alpha}$ corresponds to recognizability using the stronger parallel encoding. The precise definitions of these two different notions of recognizability are given in Section 3 of this paper.\newline

\noindent Following the argument in \cite{BBB10} we want to reduce Theorem A to the special cases that $X\subseteq \Z^d$ or $X \subseteq [0,1]^d$.  In \cite[Section 4.1]{BBB10} the corresponding argument uses special properties of $k$- and $\ell$-recognizable sets. In \cref{definbounded} we will establish  a more general definability criterion and use it instead. This is closely connected to recent work of Bès and Choffrut \cite{BesC1, BesChoffrut2022}. To reduce Theorem A to Theorem B, it is left to show that Theorem A holds when $X\subseteq \Z^d$. This is achieved in Section 6 by adjusting the main argument from \cite{Bes-CS}.\newline

\noindent The proof of Theorem B itself is similar to the one for integer bases in 
\cite{BBB10,BBL09}. However, in \cite[Lemma 6.3]{BBB10} ultimate periodicity of certain sufficiently tame sets is obtained from Cobham's theorem. Although we prove a similar result in \cref{cobhamsemenov}, the reduction does not work in our case as the set of natural numbers cannot be described easily in terms of $\alpha$-power representations when $\alpha$ is irrational. Thus, a more complicated argument for periodicity is required, given here in the proof of \cref{prodstabultper}. It is also worth pointing out that in the case of integer bases Cobham's theorem for multiplicatively independent bases is only obtained for weakly recognizable sets \cite[Theorem 3.3]{BBB10}, and the result for general recognizable sets \cite[Theorem 3.4]{BBB10} needs the stronger assumption that the integer bases have different sets of prime factors.  If $r,s\in\N_{>1}$ are multiplicatively independent, then $s^{-1}$ has an ultimately periodic $r$-representation, and the assumption on prime factors ensures that the period lengths of $s^{-n}$ are unbounded (see \cite[Lemma 6.6]{BBB10}). In contrast, when $\Q(\alpha)\ne\Q(\beta)$, then $\beta^{-1}$ has no ultimately periodic $\alpha$-representation, simplifying the proof of our analogue of \cite[Property 6.7]{BBB10}.
Our assumption $\Q(\alpha)\ne\Q(\beta)$ also is strictly stronger than multiplicative independence, but optimal in the sense that no version of our result holds for $\Q(\alpha)=\Q(\beta)$.\newline

\noindent One last comment about the proof: the analogue of Sem\"enov's theorem for subsets of $\R^n$ is proved in \cite{BBL09}, using ideas from \cite{BBB10}. We do not follow the argument in \cite{BBL09}, but rather  combine \cite[Section 3.2]{BBL09} more directly with the argument in \cite{BBB10} in our \cref{definmultidim}.

\subsection*{Acknowledgments}  P.H. and S.M. were partially supported by the Hausdorff Center for Mathematics at the University of Bonn (DFG EXC 2047). P.H. and C.S. were partially supported by NSF grant DMS-1654725. We thank Alexis Bès for answering our question.

\section{Preliminaries}

In this section we recall preliminary results from automata theory, number theory and logic. Before we do so, we fix some notation.\newline

\noindent Let $X$ be a set. We write $\#X$ for the cardinality of $X$.
Let $\Sigma$ be an alphabet. Given a (possibly infinite word) $w$ over $\Sigma$, we write $w_i$ for the $i$-th letter of $w$, and $|w|$ for the length of $w$. We let $\Sigma^*$ denote the set of finite words over $\Sigma$, and $\Sigma^{\omega}$ the set of infinite words over $\Sigma$. 
For $u_1,\dots, u_n\in \Sigma^{\omega}$,
we define the \textbf{convolution}  $c(u_1,\dots,u_n)$ as the element of $(\Sigma^n)^{\omega}$ whose value at position $i$ is the $n$-tuple consisting of the values of $u_1,\dots, u_n$ at position $i$.
For $i \in \{1,\dots,n\}$, the projection $\pi_i : (\Sigma^n)^{\omega} \to \Sigma^{\omega}$ is the function that maps $c(u_1,\dots,u_n)$ to $u_i$.

\subsection{$\omega$-regular languages}
We recall some well-known definitions and results about $\omega$-regular languages. Proofs can be found for example in Khoussainov and Nerode \cite{aut_theory}.\\

\noindent A (non-deterministic) \textbf{Muller automaton} is a quintuple $\mc A=(Q,\Sigma,T,I,F)$ with $Q$ a finite set of states, $\Sigma$ a finite alphabet, $T\subseteq Q\times\Sigma\times Q$ a transition relation, $I\subseteq Q$ a set of initial states, and $F\subseteq 2^Q$ a set of acceptance conditions. Instead of $(q_1,s,q_2)\in T$ we also write $q_1\xto{s}q_2$. An infinite word $w=w_0w_1\cdots\in\Sigma^\omega$ is \textbf{accepted} by $\mc A$ if there is a sequence $(q_n)_{n\in \N}\in Q^\N$ of states such that $q_0\in I$, and $q_n\xto{w_n}q_{n+1}$ for all $n\in \N$, and $\set{q\in Q : \set{n:q_n=q} \text{ is infinite}}\in F$. The language $L(\mc A)$ accepted by $\mc A$ is the set of words accepted by $\mc A$.\\

\noindent Let $\mc A=(Q,\Sigma,T,I,F)$ be a Muller automaton. We say $\mc A$ is \textbf{Büchi automaton} if there is $C\subseteq Q$ such that
\[
F=\set{B\subseteq Q: C\subseteq B}.
\]
We say $\mc A$ is \textbf{deterministic} if $\#I\le1$ and for all $p\in Q,s\in\Sigma$ there is at most one $q\in Q$ with $p\xto{s}q$. We say $\mc A$ is \textbf{total} if $\#I\ge1$ and for all $p\in Q,s\in\Sigma$ there is at least one $q\in Q$ with $p\xto{s}q$.\\

\noindent Let $\Sigma$ be an alphabet.  An \textbf{$\omega$-language} $K$ is a subset of $\Sigma^\omega$. For $L\subseteq\Sigma^*$, we define the $\omega$-language
\[
L^\omega:=\set{v_1v_2v_3\cdots \ : \ v_i\in L\text{ and }\abs{v_i}>0 \text{ for each $i\in \N$}},LK:=\set{vw\ : \ v\in L,w\in K}.
\]
\begin{fact}\label{condomegreg}
   Let $L$ be an $\omega$-language. Then the following are equivalent:
    \begin{enumerate}
    \item There is a Muller automaton $\mc A$ with $L(\mc A)=L$.
    \item There is a total, deterministic Muller automaton $\mc A$ with $L(\mc A)=L$.
    \item There is a total Büchi automaton $\mc A$ with $L(\mc A)=L$.
    \item There are regular languages $K_i,L_i$ with $L=\bigcup_{i=1}^nK_iL_i^\omega$.
    \end{enumerate}
    In this case, $L$ is called $\omega$-\textbf{regular}.
\end{fact}

\noindent Of crucial importance for this paper is the fact that $\omega$-regular languages are closed under the usual first-order logical operations.

\begin{fact}\label{fact:buechi}
    The collection of $\omega$-regular languages is stable under boolean combinations, convolutions and projections.
\end{fact}
\begin{fact}\label{regtestultper}
    Let $K,L$ be $\omega$-regular languages over $\Sigma$. If $K,L$ contain the same ultimately periodic words, then $K=L$.
\end{fact}
\begin{proof}
    The case $L=\emptyset$ follows immediately from \cref{condomegreg}(4), and the general case by applying this case to the symmetric difference $K\triangle L$.
\end{proof}
\subsection{Continued fractions}
We recall some basic and well-known definitions and results about continued fractions. Except for the definition of Ostrowski representations of real numbers, all these results can be found in every basic textbook on continued fractions. We refer the reader to Rockett and Szüsz \cite{RS} for proofs of these results, simply because to the authors' knowledge it is the only book discussing Ostrowski representations of real numbers in detail. \newline

\noindent A \textbf{finite continued fraction expansion} $[a_0;a_1,\dots,a_k]$ is an expression of the form
\[
a_0 + \frac{1}{a_1 + \frac{1}{a_2+ \frac{1}{\ddots +  \frac{1}{a_k}}}}
\]
For a real number $\alpha$, we say $[a_0;a_1,\dots,a_k,\dots]$ is \textbf{the continued fraction expansion of $\alpha$} if $\alpha=\lim_{k\to \infty}[a_0;a_1,\dots,a_k]$ and $a_0\in \Z$, $a_i\in \N_{>0}$ for $i>0$. For the rest of this subsection, fix a positive irrational real number $\alpha$ and assume that $[a_0;a_1,\dots,a_k,\dots]$ is the continued fraction expansion of $\alpha$.

\begin{defn}\label{def:beta} Let $k\geq 1$. We define $p_k/q_k \in \Q$ to be the \textbf{$k$-th convergent of $\alpha$}, that is the quotient $p_k/q_k$ where $p_k\in \N$, $q_k\in \Z$, $\gcd(p_k,q_k)=1$ and 
\[
\frac{p_k}{q_k} = [a_0;a_1,\dots,a_k].
\]
The \textbf{$k$-th difference of $\alpha$} is defined as $\beta_k := q_k\alpha - p_k$. We define $\zeta_k \in \R$ to be the \textbf{$k$-th complete quotient of $\alpha$}, that is
$\zeta_k = [a_k;a_{k+1},a_{k+2},\dots]$.
\end{defn}

\noindent Maybe the most important fact about the convergents we will use, is that both their numerators and denominators satisfy the following recurrence relation.

\begin{fact}{\cite[Chapter I.1 p. 2]{RS}}\label{fact:recursive} Let $q_{-1} := 0$ and $p_{-1}:=1$. Then $q_{0} = 1$, $p_{0}=a_0$ and for $k\geq 0$,
\begin{align*}
q_{k+1} &= a_{k+1} \cdot q_k + q_{k-1}, \\
p_{k+1} &= a_{k+1} \cdot p_k + p_{k-1}. \\
\end{align*}
\end{fact}
\noindent We directly get that $\beta_{k+1} = a_{k+1} \beta_k + \beta_{k-1}$ for for $k\geq 0$. We will now introduce a numeration system due to Ostrowski \cite{Ost}.

\begin{fact}[{\cite[Chapter II.4  p. 24]{RS}}]\label{ostrowski} Let $N\in \N$. Then $N$ can be written uniquely as
\[
N = \sum_{k=0}^{n} b_{k+1} q_{k},
\]
where $n\in \N$ and the $b_k$'s are in $\N$ such that $b_1<a_1$ and for all $k\in \N_{\leq n}$, $b_k \leq a_{k}$ and, if $b_k = a_{k}$, then $b_{k-1} = 0$.
\end{fact}

\noindent We call the representation of a natural number $N$ given by \cref{ostrowski} the normalized \textbf{$\alpha$-Ostrowski representation} of $N$. Of course, we will drop the reference to $\alpha$ whenever $\alpha$ is clear from the context. If $\phi$ is the golden ratio, the $\phi$-Ostrowski representation is better known as the \textbf{Zeckendorf representation}, see Zeckendorf \cite{Zeckendorf}. We will also need a similar representation of a real number.

\begin{fact}[{\cite[Chapter II.6  Theorem 1]{RS}\label{ostrowskireal}}]
Let $c \in \R$ be such that $-\frac{1}{\zeta_1} \leq c < 1-\frac{1}{\zeta_1}$. Then $c$ can be written uniquely in the form
\[
c = \sum_{k=0}^{\infty} b_{k+1} \beta_{k},
\]
where $b_k \in \N$, $0 \leq b_1 < a_1$, $0 \leq b_k \leq a_{k}$, for $k> 1$, and $b_k = 0$ if $b_{k+1} = a_{k+1}$, and $b_k < a_{k}$ for infinitely many odd $k$.
\end{fact}
\noindent We call the representation of $c$ given by \cref{ostrowskireal} the normalized \textbf{$\alpha$-Ostrowski representation} of $c$. 

\begin{fact}[{\cite[Fact 2.13]{H-Twosubgroups}}]
    \label{real_ostrowski_order}
    Let $b,c \in  [-\frac{1}{\zeta_1}, 1-\frac{1}{\zeta_1})$ be such that $b\neq c$, and  let $b_1,b_2,\dots,c_1,c_2,\ldots\in \N$ be 
    such that $ \sum_{k=0}^{\infty} b_{k+1} \beta_{k}$ and $\sum_{k=0}^{\infty} c_{k+1} \beta_{k}$ are the $\alpha$-Ostrowski representations of $b$ and $c$ respectively. Let $k \in \N_{>0}$ be minimal such that $b_k \neq c_k$. Then $b<c$ if and only if
     \begin{itemize}
    \item[(i)] $b_k < c_k$ if $k$ is odd;
    \item[(ii)] $b_k > c_k$ if $k$ is even.
    \end{itemize}
\end{fact}

\subsection{Quadratic irrationals}
Now suppose that $\alpha$ is quadratic.
For $k\in \N$, set $\Gamma_k:=\begin{bmatrix}a_k&1\\1&0\end{bmatrix}$. Note that by \cref{fact:recursive}
\[
\prod_{k=0}^m\Gamma_k=\begin{bmatrix}p_{m}&p_{m-1}\\q_{m}&q_{m-1}\end{bmatrix}.
\]
 By Lagrange's theorem (see \cite[Theorem III.2]{RS}) the continued fraction expansion of $\alpha$ is ultimately periodic. Let $P(\alpha)\in \N$ be the minimal element of $\N$ such that there is $N(\alpha)\in \N$ with  $a_n=a_{n+P(\alpha)}$ for all natural numbers $n\ge N(\alpha)$. Set 
\[
\Gamma_{\alpha}:=\prod_{k=N(\alpha)}^{N(\alpha)+P(\alpha)-1}\Gamma_k.
\]
Observe that $\det \Gamma_{\alpha}\in \{-1,1\}$.
Let $\mathcal{O}_\alpha$ be the ring of integers of $\Q(\alpha)$.\newline

\noindent We now collect two surely well-known facts about $\Gamma_{\alpha}$. Since we did not find exact references, we include the proofs for the convenience of the reader.

\begin{fact}\label{charpoly}
The characteristic polynomial of $\Gamma_{\alpha}$ is the minimal polynomial of an element of $\mathcal{O}_{\alpha}^\times$, whose roots have distinct absolute values. The eigenvalue with greatest absolute value is greater than $1$.
\end{fact}
\begin{proof}
  Since replacing $\alpha$ by some $[0;a_k,a_{k+1},\dots]$ does neither change $\Q(\alpha)$ nor $\Gamma_{\alpha}$, we can assume that the continued fraction expansion is purely periodic. Therefore
\[
\Gamma_{\alpha} = \prod_{k=0}^{P(\alpha)-1}\Gamma_k=\begin{bmatrix}p_{P(\alpha)-1}&p_{P(\alpha)-2}\\q_{P(\alpha)-1}&q_{P(\alpha)-2}\end{bmatrix}.
\]
Thus the characteristic polynomial of $\Gamma_{\alpha}$ is 
\[
X^2 - (q_{P(\alpha)-2}+p_{P(\alpha)-1})X + (q_{P(\alpha)-1}p_{P(\alpha)-2}-q_{P(\alpha)-2}p_{P(\alpha)-1}).
\]
As in the proof of \cite[Ch. III, §1, Thm. 1]{RS} we conclude that $\alpha$ is a root of
\[
  q_{P(\alpha)-1}X^2+(q_{P(\alpha)-2}-p_{P(\alpha)-1})X-p_{P(\alpha)-2}=0.\]
The discriminant of this polynomial equals
\[(q_{P(\alpha)-2}-p_{P(\alpha)-1})^2+4q_{P(\alpha)-1}p_{P(\alpha)-2}=(q_{P(\alpha)-2}+p_{P(\alpha)-1})^2+4(q_{P(\alpha)-1}p_{P(\alpha)-2}-q_{P(\alpha)-2}p_{P(\alpha)-1}).
\]
Thus both polynomials have the same nonzero discriminant. Hence the characteristic polynomial of $\Gamma_{\alpha}$ has two distinct roots, both in $\Q(\alpha)$. Since $-(q_{P(\alpha)-2}+p_{P(\alpha)-1})<0$, these two roots have distinct absolute values and the greater one is positive.
\end{proof}

\noindent In the following, $\norm{\Gamma_\alpha}$ denotes the operator norm $\sup_{\norm v=1}\norm{\Gamma_\alpha v}$. Note that it equals the largest absolute value of an eigenvalue of $\Gamma_\alpha$. 
\begin{fact}\label{shifting}
Let $k \in \N$ be such that $k<P(\alpha)$. Then there are $m\in\N$ and $C,D,E\in\Q(\alpha)^\times$ such that for all $n\in\N_{\ge m}$
\begin{align*}
q_{nP(\alpha)+k} &= C\norm{\Gamma_{\alpha}}^n+D(\det \Gamma_{\alpha}\norm{\Gamma_{\alpha}})^{-n},\\
\beta_{nP(\alpha)+k} &= E(\det\Gamma_\alpha\norm{\Gamma_\alpha})^{-n}.
\end{align*}
\end{fact}
\begin{proof}
Let $M\in \N$ be such that $a_n = a_{n + P(\alpha)}$ for all $n\in \N_{\ge M}$. Increasing $M$, we can assume that $M=mP(\alpha)+k$ for some $m\in \N$. Let $v_1,v_2\in \Q(\alpha)^2$ be eigenvectors of $\Gamma_{\alpha}$, corresponding to the eigenvalues $\lambda_1=\norm{\Gamma_{\alpha}} ,\lambda_2=\det \Gamma_{\alpha}\norm{ \Gamma_{\alpha}}^{-1}$, such that
\[
\begin{bmatrix}
p_{M}\\
q_{M}
\end{bmatrix} = v_1 + v_2.
\]
Then 
\[
\begin{bmatrix}
p_{M+nP(\alpha)}\\
q_{M+nP(\alpha)}
\end{bmatrix} = \Gamma_{\alpha}^n 
\begin{bmatrix}
p_{M}\\
q_{M}
\end{bmatrix}=\lambda_1^nv_1+\lambda_2^nv_2. 
\]
Note that for $n\ge m$
\[
\begin{bmatrix}
p_{nP(\alpha)+k}\\
q_{nP(\alpha)+k}
\end{bmatrix} = \Gamma_{\alpha}^{-m}
\begin{bmatrix}
p_{M+nP(\alpha)}\\
q_{M+nP(\alpha)}
\end{bmatrix} = \lambda_1^n\Big(\Gamma_{\alpha}^{-m}v_1\Big) +\lambda_2^n\Big(\Gamma_{\alpha}^{-m}v_2\Big).
\]
Let $C$ be the second coordinate of $\Gamma_{\alpha}^{-m}v_1$ and let $D$ be the second coordinate of $\Gamma_{\alpha}^{-m}v_2$. Since $q_{nP(\alpha)+k}\to\infty$ as $n\to \infty$, we have that $C\ne0$. Hence $\{C\lambda_1^n\ : \ n\in \N\}\nsubseteq\Q$ by irrationality of $\lambda_1$. Since $q_{nP(\alpha)+k}\in\Q$, we conclude $D\ne0$.\\
The proof for $\beta_{nP(\alpha)+k}$ is analogous, noting that $v_1=0$ in this case since $\abs{\beta_{nP(\alpha)+k}}\to 0$ as $n\to\infty$.
\end{proof}

\noindent We collect one final fact about quadratic numbers. It shows that the assumptions on $\alpha$ and $\beta$ in Theorem A imply the assumptions in Theorem B.

\begin{fact}\label{quadirrmulind}
  Let $\alpha,\beta \in \R_{>0}$ be quadratic irrationals with $\Q(\alpha)\ne\Q(\beta)$. Then $\alpha$ and $\beta$ are multiplicatively independent.
\end{fact}
\begin{proof}
Towards a contradiction, suppose that there are $m,n\in \N$ such that $\alpha^m=\beta^n$. Without loss of generality, we can assume that $m$ and $n$ are coprime. For $k\in\N$ write $\alpha^k=u_k+v_k\sqrt z$ with $u_k,v_k,z\in\Q$ and $v_1,z\ne0$. Then $v_m$ is a nontrivial $\N$-linear combination of
  \[u_1^{m-1}v_1,u_1^{m-3}v_1^3,\dots,u_1^{m-1-2\floor{(m-1)/2}}v_1^{2\floor{(m-1)/2}+1},\]
  which all have the same sign. Since $\alpha^m\in\Q(\alpha)\cap\Q(\beta)=\Q$, we have $v_m=0$ and so $u_1=0$. Thus $m$ is even. Analogously we see that $n$ is even. Then $m,n$ are not coprime, contradicting our assumption.
\end{proof}

\subsection{Non-definability and a theorem of Bès and Choffrut}
We now collect some definability and non-definability results in expansions of the real ordered additive group. Let $K\subseteq \R$ be a subfield, and consider the signature 
\[
\mathcal{L}_K:=\{<,+,1, (\lambda_k)_{k\in K}\},
\]
where $\lambda_k : \R \to \R$ maps $x$ to $kx$. We will consider the $\mathcal{L}_K$-structure $\Big(\R,{<},+,1,(x\mapsto \alpha x)_{\alpha \in K}\Big)$. It is well known that this structure has quantifier-elimination. It follows that every set definable in $(\R,{<},+,1,(x\mapsto \alpha x)_{\alpha \in K})$ is a finite union of open and closed $K$-polyhedra. When $K=\Q$, it is clear that $(\R,{<},+,1,(x\mapsto \alpha x)_{\alpha \in \Q})$ and $(\R,{<},+,1)$ define the same sets. The following fact shows that adding $\Z$ to $(\R,{<},+,1)$, does not add new bounded definable sets.

\begin{fact}{{\cite[Lemma 2.4]{BesC1}}}\label{bc1fact}
Let $X\subseteq \R^n$ be definable in $(\R,<,+,\Z)$, let $a_1,\dots,a_n,b_1,\dots,b_n \in \Q$. Then
\[
X \cap \Big( [a_1,b_1] \times \dots \times [a_n,b_n] \Big)
\]
is definable in $(\R,<,+,1)$.
\end{fact}

\noindent Fix $d\in \N$ for the remainder of this subsection. Let $\mathcal{L}_K(P)$ be the signature $\mathcal{L}_K$ together with an $d$-ary predicate symbol $P$. This allows us to consider for $X\subseteq \R^d$ the $\mathcal{L}_K(P)$-structure $(\R,{<},+,1,(x\mapsto \alpha x)_{\alpha \in K},X)$.\newline

\noindent The following theorem was shown for $K=\Q$ by Bès and Choffrut \cite{BesChoffrut2022}. An inspection of their proof shows that the statement holds for arbitrary subfields of $\R$.
\begin{fact}{{\cite[Theorem 5.8]{BesChoffrut2022}}}\label{selfdef}
    Let $K\subseteq\R$ be a subfield. Then there exists an $\mathcal{L}_K(P)$-sentence $\Phi_n$ such that for every $X\subseteq\R^n$ for which every nonempty set definable in $(\R,{<},+,1,(x\mapsto \alpha x)_{\alpha \in K},X)$ contains a point with components in $K$, the following are equivalent:
    \begin{enumerate}
        \item $(\R,{<},+,1,(x\mapsto \alpha x)_{\alpha \in K},X)\models \Phi_n$,
        \item  $X$ is $(\R,{<},+,1,(x\mapsto \alpha x)_{\alpha \in K})$-definable.
    \end{enumerate}
\end{fact}

\noindent Following the argument in \cite{BBB10} we want to reduce Theorem A to the special case that $X\subseteq \Z^d$ or $X \subseteq [0,1]^d$. 
In \cite[Section 4.1]{BBB10} the corresponding argument uses special properties of $k$- and $\ell$-recognizable sets. Here we will use the following general proposition.

\begin{prop}\label{definbounded}
  Let $X\subseteq\R^d$ be not definable in $(\R,{<},+,\Z)$. Then one of the following holds:
  \begin{enumerate}
  \item $(\R,{<},+,\Z,X)$ defines a subset of $[0,1]^d$ not definable in $(\R,{<},+,\Z)$.
  \item $(\R,{<},+,\Z,X)$ defines a subset of $\N^d$ not definable in $(\Z,{<},+)$.
  \end{enumerate}
  \end{prop}

\noindent To prove \cref{definbounded}, we need two elementary results about the expressive strength of $\mathcal{L}_K(P)$-formulas. We expect both (and maybe even \cref{definbounded} itself) to be known, but include proofs for the convenience of the reader.

\begin{fact}\label{lemsubspdef}
Let $k\in \{0,\dots,d\}$. Then there exists an $\mathcal{L}_\Q(P)$-sentence $\Psi_{k,d}$ such that for every $X\subseteq\R^d$, the following are equivalent:
    \begin{enumerate}
        \item $(\R,{<},+,1,(x\mapsto \alpha x)_{\alpha \in \Q},X)\models \Psi_{k,d}$,
        \item $X$ is a $k$-dimensional affine $\R$-subspace of $\R^d$.
    \end{enumerate}
    Similarly, there exists an $\mathcal{L}_\Q(P)$-formula $\Psi'_{k,d}(x)$ such that for every $X\subseteq\R^d$ and $a\in\R^d$, the following are equivalent:
    \begin{enumerate}
        \item $(\R,{<},+,1,(x\mapsto \alpha x)_{\alpha \in \Q},X)\models \Psi_{k,d}(a)$,
        \item There are a neighborhood $U$ of $a$ and a $k$-dimensional affine $\R$-subspace $V$ of $\R^d$ such that $X\cap U=V\cap U$.
    \end{enumerate}
\end{fact}
\begin{proof}
Note that $X$ is an $\R$-subspace of $\R^n$ if and only if it is closed, an additive subgroup, and stable under multiplication by $\frac12$. Thus we can easily express whether $X$ is an affine subspace, by an $\mathcal{L}_\Q(P)$-sentence. 
An $\R$-subspace $W$ of $\R^d$ has dimension at least $k$ if and only if there are  $v_1,\dots,v_k \in W$ and a strictly increasing sequence $(i_1,\dots,i_k)\in \{1,\dots,n\}^k$ such that for all $j \in \{1,\dots, k\}$ 
\[
v_{j,1}=\dots=v_{j,i_j-1}=0 \text{ and } v_{j,i_j}\ne0.
\]
Indeed, every such tuple is linearly independent, and conversely every linearly independent tuple can be transformed into such a tuple by Gaussian elimination. The set of tuples satisfying the above condition is $\mathcal{L}_\Q$-definable. Hence we can express using an $\mathcal{L}_\Q(P)$-sentence if an $\R$-subspace of $\R^n$ has dimension at least $k$. The first part of the lemma follows easily.\\
For the second part it suffices to construct such a formula for a fixed open box $U$ around the origin of diameter $\eps$. We can use the construction of the first part, just weakening the condition of being an additive subgroup to ``$a+b\in X$ whenever $a,b\in X$ and $a,b,a+b\in U$''.
\end{proof}

\begin{fact}\label{verticesdef}
  Let $K\subseteq\R$ be a subfield. Then there exists a $d$-ary $\mc L_K(P)$-formula $\Theta_d(x_1,\dots,x_d)$ such that for every bounded $X\subseteq \R^d$ definable in $(\R,{<},+,1,(x\mapsto \alpha x)_{\alpha \in K})$, the set $V\subseteq\R^d$ defined by $\Theta_d$ in $(\R,{<},+,1,(x\mapsto \alpha x)_{\alpha \in K},X)$ satisfies the following properties:
  \begin{enumerate}
    \item $V$ is finite, and
    \item $X$ is a union of $K$-polyhedra with vertices in $V$.
  \end{enumerate}
\end{fact}
\begin{proof}
We show by decreasing induction on $k\le d$ that there is an $\mc L_K(P)$-formula $\Theta_{k,d}(x_1,\dots,x_d)$ such that for every bounded set $X\subseteq \R^d$ definable in $(\R,{<},+,1,(x\mapsto \alpha x)_{\alpha \in K})$
\begin{enumerate}
    \item the set \[\{ a \in \R^d \ : \ (\R,{<},+,1,(x\mapsto \alpha x)_{\alpha \in K},X)\models \Theta_{k,d}(a)\}\] is a finite union of $\le k$-simplices, 
    \item $X$ is a union of $K$-polyhedra all whose $\le k$-faces are contained in this set.
\end{enumerate}
Then $\Theta_{0,d}(x_1,\dots,x_d)$ is the desired formula.\newline
\noindent For $k=d$, simply set $\Theta_{k,d}$ to be $\top$, using that sets definable in $(\R,{<},+,1,(x\mapsto \alpha x)_{\alpha \in K})$ are finite unions of $K$-polyhedra. Suppose now that $k\in \{0,\dots,d-1\}$ and we have constructed $\Theta_{k,d}$ with the desired property. By \cref{lemsubspdef} there is an $\mc L_K(P)$-formula $\Theta_{k-1,d}(x_1,\dots,x_d)$ such that  
for every bounded set $X\subseteq \R^d$ definable in $(\R,{<},+,1,(x\mapsto \alpha x)_{\alpha \in K})$, and every $b\in \R^d$, the following are equivalent
\begin{enumerate}
    \item $(\R,{<},+,1,(x\mapsto \alpha x)_{\alpha \in K},X)\models \Theta_{k-1,d}(b)$,
    \item there is no  $k$-dimensional affine subspace $W$ and no neighbourhood $U$ of $b$
    such that either 
    \begin{enumerate}
        \item $\{ a \in X \ : \ (\R,{<},+,1,(x\mapsto \alpha x)_{\alpha \in K},X)\models \Theta_{k,d}(a)\}\cap U=W\cap U$, or 
        \item $\{ a \in \R^d\setminus X \ : \ (\R,{<},+,1,(x\mapsto \alpha x)_{\alpha \in K},X)\models \Theta_{k,d}(a)\}\cap U=W\cap U$.
    \end{enumerate}
\end{enumerate}
We leave it to the reader to check that this formula has the desired properties.
\end{proof}

\begin{proof}[Proof of \cref{definbounded}]
Suppose (1) fails. Without loss of generality, we may assume that $X\subseteq\R_{\ge0}^d$.  By the failure of (1) and \cref{bc1fact}, we have that for each $m=(m_1,\dots,m_d)\in \N^d$, the set $X\cap\prod_{i=1}^d[m_i,m_i+1]$ is definable in $(\R,<,+,1)$. Let $W$ be the union of the sets of \cref{verticesdef}, applied to $\Q$ and $X\cap\prod_{i=1}^d[m_i,m_i+1]$, over $m\in\N^d$. Then $W$ is definable in $(\R,{<},+,\Z,X)$ and countable. Set
  \[
  B=\set{(x_1-\floor{x_1},\dots,x_d-\floor{x_d})\ : \ x\in W}.
  \]
   Note that $B$ is a countable subset of $[0,1]^d$ definable in $(\R,{<},+,\Z,X)$. By the failure of (1), the set $B$ is definable in $(\R,{<},+,\Z)$, and hence a finite subset of $\Q^d$. Set
   \[
   C:=\set{\frac1k\sum_{i=1}^kb_i:1\le k\le d+1\text{ and }b_i\in B}.
   \]
   Note that each simplex occurring in $(X-(m_1,\dots,m_d))\cap[0,1]^d$ contains an element of $C$ as interior point. This follows immediately as its vertices lie in $B$. Thus $(X-(m_1,\dots,m_d))\cap[0,1]^d=(X-(m_1',\dots,m_d'))\cap[0,1]^d$ if and only if this equality holds after intersecting with $C$.\\
  Hence the equivalence relation thus defined has finitely many equivalence classes and is definable in $(\R,{<},+,\Z,X)$. It can not be definable in $(\R,{<},+,\Z)$, because $X$ is not. Thus one of its finitely many equivalence classes is also not definable in $(\R,{<},+,\Z)$, thus witnessing (2).
\end{proof}

\noindent We finish this subsection by introducing some notation and tools that allow us to track local non-definability of sets.

\begin{defn}
 Let $S\subseteq \R^d$, and let $x \in \R^d$. We say $x$ is a \textbf{bad point with respect to $S$} if for every neighborhood $U$ of $x$, the set $S\cap U$ is not definable in $(\R,{<},+,1)$.
\end{defn}

\begin{lem}\label{zerobad}
Let $X\subseteq\R^d$, and let $q=(q_1,\dots,q_d) \in \Q^d$ be a bad point with respect to $X$. Then $(\R,{<},+,1,X)$ defines a set $X'\subseteq[0,1]^d$ such that $0$ is a bad point with respect to $X'$.
\end{lem}
\begin{proof}
  Throughout, definable means definable in $(\R,{<},+,1)$. By translating, we reduce to the case that $X\subseteq[0,1]^d$. Set
  \[
  A_k := \{(x_1,\dots,x_d) \in \R^d \ : \  k_ix_i\le k_iq_i\},
  \]
  and observe that $\R^d=\bigcup_{k\in\set{-1,1}^d} A_k$. We now show that there is $k_0\in\set{-1,1}^d$ such that $q$ is a bad point with respect to $X\cap A_{k_0}$. Indeed, suppose $q$ is not a bad point with respect to $X \cap A_k$ for $k\in\set{-1,1}^d$. Thus for each $k\in\set{-1,1}^d$ we pick an open neighborhood $U_k$ of $q$ such that $(X\cap A_k) \cap U_k$ is definable. By intersecting with an open box whose corners have rational coordinates, we may assume that each $U_k$ is definable.
  Thus for each $k \in \set{-1,1}^d$ the set
    \[
    (X\cap A_k) \cap \bigcap_{k'\in\set{-1,1}^d}U_{k'}
    \]
    is definable. Since the $A_k$'s cover $\R^d$, we conclude that $X\cap\bigcap_{k'}U_{k'}$ is definable. This contradicts that $q$ is a bad point with respect to $X$.\par
  \noindent We may that assume $k_0=(-1,\dots,-1)$ by applying a reflection along a coordinate axis. Since $q$ is a bad point with respect to $X\cap A_{k_0}$, we get by translation that $0$ is a bad point with respect to 
  \[
  (X\cap A_{k_0})-q=(X-q)\cap[0,1]^d.
  \]
\end{proof}

\section{Regularity and numeration systems}

In this section, we introduce an abstract notion of a numeration system for real numbers that is convenient for our purposes and allows us to prove equivalence of recognizability and definability for such a system. 

As mentioned in the introduction, we not only want to give a uniform treatment of a large class of numeration systems, we also need to address the differences that arise from the choice of encoding the integer part and fractional part of a real number either parallely or sequentially. In this section, we define the correct notions of parallel and sequential regularity for extended $\omega$-words.

\begin{defn}
An \textbf{extended $\omega$-word} is a sequence $(x_i)_{i\in \Z}\in\N^\Z$ such that
\begin{enumerate}
    \item there is $N\in \Z$ with $x_i=0$ for all $i\in \Z_{\geq N}$, and 
    \item $\{x_i \ : \ i\in \Z\}$ is bounded.
\end{enumerate}
Let $\mathcal{W}$ denote the set of extended $\omega$-words, and  for each $M\in\N$ we let $\mc W_M$ denote the set of extended $\omega$-words $(x_i)_{i\in \Z}$ with $x_i\le M$ for all $i\in\Z$. We say that $X\subseteq\mc W$ is \textbf{bounded} by $M$ if $X\subseteq\mc W_M$.
\end{defn}

\noindent When we write $x_Nx_{N-1}\cdots x_0\star x_{-1}x_{-2}\cdots$, where $N\in \N$ and $x_i\in \N$ for all $i\in \Z_{\leq N}$, we mean the extended $\omega$-word $(w_i)_{i\in \Z}$ such that
\[
w_i = \begin{cases}
x_i & i \leq N \\
0 & \, \text{otherwise.}
\end{cases}
\]
Note that every extended word is of this form, and the representation is unique up to choosing $N$. We call the subword $x_Nx_{N-1}\cdots x_0$ in this representation the \textbf{integral part} and the infinite subword $x_{-1}x_{-2}\cdots$ the \textbf{fractional part} of $(w_i)_{i\in \Z}$.
\begin{defn}
Let $X\subseteq\mathcal{W}^n$ be bounded by $M$. We call
    \begin{multline*}
      \{((x^1_N,\dots,x^n_N)\cdots(x^1_0,\dots,x^n_0)\star (x^1_{-1},\dots,x^n_{-1})\cdots)\in(\Sigma_M^n\cup\{\star\})^\omega\\
      \ : \ (x^1_N\cdots x^1_0\star x^1_{-1}\cdots,\dots,x^n_N\cdots x^n_0\star x^n_{-1}\cdots)\in X\}
    \end{multline*}
    the \textbf{sequential representation} of $X$, and we say $X$ is \textbf{sequentially regular} if its sequential representation is $\omega$-regular.  We call
        \begin{multline*}
      \{((x^1_0,x^1_{-1},x^2_0,\dots,x^n_0,x^n_{-1})(x^1_1,x^1_{-2},x^2_1,\dots,x^n_1,x^n_{-2})\cdots)\in(\Sigma_M^{2n})^\omega\\
      \ : \ (x^1_N\cdots x^1_0\star x^1_{-1}\cdots,\dots,x^n_N\cdots x^n_0\star x^n_{-1}\cdots)\in X\}
    \end{multline*}
    the \textbf{parallel representation} of $X$, and we say $X$ is \textbf{parallelly regular} if its parallel representation is $\omega$-regular. A relation or function is sequentially (parallelly) regular if its graph is.
\end{defn}
\begin{prop}\label{prop:srimppr}
Let $X\subseteq \mathcal{W}^n_M$ be sequentially regular. Then $X$ is parallelly regular.
\end{prop}
\noindent The converse is easily seen to fail: simply consider the set with sequential representation 
\[
\{ a_n \cdots a_2a_1a_0\star a_0a_1 \cdots a_n0^{\omega} : a_0a_1\cdots a_n \in \Sigma_M^*\}.
\]
This set is parallelly regular, but not sequentially regular.

\begin{proof}[Proof of \cref{prop:srimppr}]
Let $\mc A=(Q,\Sigma_M^n\cup\{\star \},T,I,F)$ be a Büchi automaton recognizing the sequential representation of $X$. Let $\mc A_{p_1,p_2}=(Q,\Sigma_M^n\cup\{\star \},T_{p_1,p_2},I,F)$ be the Büchi automaton obtained from $\mc A$ such that
\[
T_{p_1,p_2} = T \setminus \{ (q_1,\star ,q_2) \ : (q_1,q_2) \ne (p_1,p_2)\}.
\]
Since each accepting computation contains a unique transition of the form $p_1\xto\star p_2$, the accepted language of $\mc A$ is the finite union over the languages accepted by $\mc A_{p_1,p_2}$ for $p_1\xto\star p_2$ in $\mc A$. Thus we may assume that $\mc A$ has a unique such transition.\par

\noindent In this case define the Büchi automaton
\[\mc A'=(Q^2 \times \{0,1\},\Sigma_M^{2n},T',\{(p_1,p_2,0)\},I\times F\times\set{0,1}),\]
where 
\begin{align*}
T' &= \{ ((q_1,q_2,0),(s_1,s_2),(q_3,q_4,i)) : q_1\xto{s_1}q_3 \text{ and } q_4\xto{s_2}q_2 \}\\
& \cup \{((q,q_2,1),(0,s_2),(q,q_4,1)) :  q_2\xto{s_2}q_4 \}.
\end{align*}
So, we simulate $\mc A$ on the first coordinate backwards from $p_1$, recognizing the integral part, and on the second coordinate forwards from $p_2$, recognizing the fractional part. The transition moves from one copy of $Q^2$ to the other when the integral part has ended. It is clear from the construction that $\mc A'$ recognizes the parallel representation of $X$.
\end{proof}

\subsection{Numeration systems} We are now ready to introduce a notion of abstract numeration systems.
\begin{defn}
\begin{enumerate}
    \item A \textbf{pre-numeration system} is a sequence $U=(U_i)_{i\in \Z} \in\R^\Z$ such that
    \begin{enumerate}
        \item $|U_i|<|U_j|$ for all $i,j \in \Z$ with $i<j$, and 
        \item for all $M\in\N$ and $w\in \mathcal{W}_M$ the series $\sum_iw_iU_i$ converges.
    \end{enumerate}
     \item We define 
    \begin{align*}
     [-]_U\colon\bigcup_M\mc W_M&\to\R,\\   
     w&\mapsto\sum_iw_iU_i.
    \end{align*}
Similarly, if $w$ is just an $\omega$-word, we define $[w]_U$ as the value of the corresponding extended word, i.e., $[w]_U:=[\star w]_U$, and if $w$ is a finite word, we set $[w]_U:=[w\star 0^\omega]_U$.
    \item A \textbf{$U$-representation} of $x\in\R$ is an extended word $w$ with $[w]_U=x$.
\end{enumerate}
\end{defn}

\begin{defn}
A \textbf{numeration system} is a triple $\mathcal{S}=(M,U,\rho)$ consisting of  
  \begin{enumerate}
    \item a natural number $M$,
    \item a pre-numeration system $U$, and
    \item a right inverse $\rho: \R_{\ge0}\to \mathcal{W}_M$ of the restriction of $[-]_U$ to $\mc W_M$.
  \end{enumerate}
  We write $[-]_{\mc S}$ for $[-]_U$, and an \textbf{$\mathcal{S}$-representation} of $x\in\R_{\ge0}$ is an $U$-representation of $x$.
\end{defn}

\noindent Note that an $\mc S$-representation of a non-negative real number is not unique.

\begin{defn}
 Let $\mathcal{S}=(M,U,\rho)$ be a numeration system. 
 We say $w \in\mathcal{W}_M$ is \textbf{$\mathcal{S}$-normalized} if $w$ is in the image of $\rho$. The \textbf{$\mathcal{S}$-normalization} of $w$ is $\rho([w]_{\mc S})$.\newline
 We define $\N_{\mathcal{S}}$  as the set of numbers $x\in \R_{\ge0}$ such that $\rho(x)=y\star 0^\omega$ for some $y\in\set{0,\dots,M}^*$. \newline
 We say $\mathcal{S}$ is \textbf{greedy} if for all $w \in\mathcal{W}_M$, the following are equivalent:
 \begin{enumerate}
     \item $w$ is $\mathcal{S}$-normalized,
     \item $\sum_{i=-\infty}^{j-1} \rho(x)_i U_i < U_j$ for all $j\in \Z$.     
\end{enumerate}
\end{defn}
\begin{ex}
    Let $\beta\in \R_{>1}$. Clearly, $U_{\beta}:=(\beta^i)_{i\in\Z}$ is a pre-numeration system. We extend it to a greedy numeration system $\mc S_\beta$, which we call the \textbf{power numeration system based on $\beta$}: set $M_{\beta}:=\ceil{\beta}-1$, and let $\rho_{\beta}: \R_{\geq 0} \to \mathcal{W}_M$ map $x\in \R_{\geq 0}$ to the lexicographically maximal $w\in\mc W_M$ with $[w]_{U\beta}=x$. The $\mathcal{S}_{\beta}$-representations of a positive number are precisely its $\beta$-representation as introduced in \cite{Renyi}, and its $\beta$-expansion (as in \cite[Section 2.1]{CLR}) is its normalized $\mathcal{S}_{\beta}$-representation.
\end{ex}

\noindent Of course, $\mc S_{10}$ is the usual decimal numeration system, that is, $\rho(x)$ is the unique decimal representation of $x \in \R_{\geq 0}$ that does not end in $9^{\omega}$.

\begin{ex}
Let $\alpha$ be a quadratic irrational number. The continued fraction expansion $$[a_0;a_1,\dots,a_k,\dots]$$ of $\alpha$ is ultimately periodic, and in particular bounded. 
  The \textbf{Ostrowski numeration system} $\mathcal O_{\alpha}=(M^{\alpha},U^{\alpha},\rho^{\alpha})$ with respect to $\alpha$ is defined as follows:
  \begin{enumerate}
      \item $M^{\alpha} := \max_{k\in \N} a_k$
      \item $U^\alpha_i:=\begin{cases}
			q_{i+1}, & \text{if $i\geq 0$,}\\
            \beta_{-i}, & \text{otherwise,}
            		 \end{cases}$
    \item $\rho^{\alpha}: \R_{\ge0}\to \mathcal{W}_{M_{\alpha}}$ maps $x$ to the unique $b=(b_i)_{i\in \Z} \in \mathcal{W}_{M_{\alpha}}$ such that
    \begin{enumerate}
        \item $x=[b]_{U^{\alpha}},$
        \item $b_0,b_{-1}<a_1,b_{k+1},b_{-k}\le a_k$, and $b_{k+1}=a_k$ (resp. $b_{-k}=a_k$) implies $b_k=0$ (resp. $b_{-k+1}=0$), for all $k\in \N$, and 
        \item  $b_{-k}<a_k$ for infinitely many odd $k\in \N$. 
    \end{enumerate}
  \end{enumerate}
\noindent Well-definedness of $\rho^{\alpha}$ follows from \cref{ostrowski,ostrowskireal}.
\end{ex}

\subsection{Feasible numeration systems}
\noindent So far, we have imposed nearly no conditions on our numeration systems. However, the numeration system we wish to consider all satisfy strong regularity conditions. 
 \begin{defn} Let $\mathcal{S}=(M,U,\rho)$ be a numeration system. We say $\mathcal{S}$ is \textbf{feasible} if
    \begin{enumerate}
    \item the $\mc S$-normalization map $\mc W_M \to \mc W_M$ sending $w$ to $\rho([w]_{\mc S})$ is sequentially regular,
    \item the relation 
    \[
    \{(w_1,w_2)\in \mathcal{W}_M^2 \ : \ [w_1]_{\mc S}<[w_2]_{\mc S}\}
    \]
    is sequentially regular,
    \item there is $k\in \N$ such that every word of the form 
    \[a_n0^ka_{n-1}0^k\cdots a_00^k\star a_{-1}0^ka_{-2}0^k\cdots\]
    with $a_i\in\{0,1\}$ is $\mc S$-normalized.
    \end{enumerate}

\end{defn}

\begin{lem}
  Let $\mathcal{S}=(M,U,\rho)$ be a greedy numeration system. Then $\mc S$ is feasible if and only if $\mc S$-normalization is sequentially regular and there is $k\ge0$ such that every word $10^k\cdots 10^k\star 10^k\cdots$ is $\mc S$-normalized.
\end{lem}
\begin{proof}
  On normalized words, the order $<$ is given by lexicographic ordering, which is clearly sequentially regular. Replacing the ones in the word by $a_i$ preserves normalizedness by definition of greediness.
\end{proof}
\begin{cor}\label{powerfeas}
Let $\beta \in \R_{>1}$. Then $\mathcal S_\beta$ is feasible if and only if $\mc S_{\beta}$-normalization is sequentially regular.
\end{cor}
\begin{proof}
    It suffices to show that there exists $k\in \N$ such that $(10^k)^n\star(10^k)^\omega$ is $\mc S_{\beta}$-normalized for all $n\in \N$. Suppose not.  Let $k\in \N$ be such that $\beta^{-k}<\frac12$.  By our assumption there is $j\in \Z$ such that $\sum_{i=-\infty}^{j-1}\beta^{ki}\ge\beta^{kj}$. By shifting we may assume that $j=0$. We obtain a contradiction, since
    \[
    1>(1-\beta^{-k})^{-1}-1=\sum_{i=-\infty}^{-1}\beta^{ki}\ge1.\qedhere
    \]
\end{proof}
\noindent If $\beta$ is a Pisot number, then $\mc S_{\beta}$-normalization is sequentially regular by Frougny \cite[Corollary 3.4]{Frougny92}.
\begin{cor}\label{gammafeas}
  Let $\beta \in \R_{> 1}$ be a Pisot number. Then $\mathcal{S}_{\beta}$ is feasible.
\end{cor}

\subsection{Recognizability} 
We are now ready to formally define recognizability of subsets of $\R_{\geq 0}^n$.
\begin{defn}
  Let $\mathcal{S}=(M,U,\rho)$ be a numeration system.   We say $X\subseteq\R_{\geq 0}^n$ is \textbf{sequentially (parallelly) $\mathcal{S}$-recognizable}, if $\rho(X)$ is sequentially (parallelly) regular.
\end{defn}

\noindent It is easy to check that for $\beta\in \R_{>1}$, sequential $\mathcal{S}_{\beta}$-recognizability corresponds to $\beta$-recognizability as defined in \cite{CLR}. It is important to point out that we restrict ourselves to subsets of the non-negative real numbers in order to keep the exposition as simple as possible. See for example \cite[p.91]{CLR} or \cite[Section II]{CSV} for ways of extending our encodings to negative numbers.

\begin{prop} Let $\mathcal{S}=(M,U,\rho)$ be a feasible numeration system, and let $X\subseteq \R_{\geq 0}^n$. Then the following are equivalent:
\begin{enumerate}
    \item $X$ is sequentially (parallelly) $\mathcal{S}$-recognizable,
    \item
    $
    \{(w_1,\dots,w_n)\in \mathcal{W}_M^n\ : \ ([w_1]_{\mc S},\dots,[w_n]_{\mc S})\in X\}$
    is sequentially (parallelly) regular.
\end{enumerate}
\end{prop}
\begin{proof}
    The stated set is the preimage of $\rho(X)$ under the $\mc S$-normalization map.
\end{proof}
\begin{lem}\label{addnormregular}
    Let $\mc S=(M,U,\rho)$ be a numeration system such that there is $k\ge0$ such that every word of the form $a_n0^k{a_{n-1}}0^k\cdots a_00^{k-\ell}\star 0^\ell a_{-1}0^ka_{-2}0^k\cdots$ with $a_i\in\set{0,1}$ and $0\le\ell\le k$ is $\mc S$-normalized. Then the following are equivalent:
    \begin{enumerate}
        \item the graph of addition $\{ (x,y,z) \in \R_{\geq0}^3 \ : \ x+y=z\}$ is sequentially $\mc S$-recognizable,
        \item $\mc S$-normalization is sequentially regular.
    \end{enumerate}
\end{lem}
\begin{proof}
 Suppose that $\mc S$-normalization is sequentially regular. Then the graph of addition is sequentially $\mc S$-recognizable, as addition may be performed by first adding componentwise two elements in $\mc W_M$, which is clearly sequentally regular, and then applying $\mathcal{S}$-normalization.\newline

 \noindent Conversely, first note that
 \[
B := \{ (\rho(x_1),\dots,\rho(x_k),\rho(y)) \in \mathcal{W}_M^{k+1} \ : \ x_1,\dots,x_k,y\in \R_{\geq 0}, \ y = \sum_{i=1}^k x_i \}
 \]
is sequentially regular.  
Observe that there is a map $f=(f_1,\dots,f_k): \mathcal{W}_M \to \mathcal{W}_M^k$ such that for each $i=1,\dots,k$, and $w \in \mc{W}_m$
 \begin{itemize}
     \item the coordinate function $f_i$ is sequentially regular,
     \item $f_i(w)$ is  of the form stated in the assumption, and
     \item $\rho(w) = \sum_{i=1}^k \rho(f_i(w))$.
      \end{itemize}
Since $f_i(w)$ is $\mc S$-normalized for each $w \in \mc{W}_M$ and $i\in \{1,\dots,k\}$, the $\mc S$-normalization of $w$ is the unique $v \in \mc{W}_M$ such that $(\rho(f_1(w)),\dots,\rho(f_k(w)),v) \in B$. Since $B$ and the graph of $f$ are sequentially regular, it follows that $\mc S$-normalization is sequentially regular as well.
\end{proof}
\begin{cor}\label{addregular}
Let $\mc S=(M,U,\rho)$ be a feasible numeration system. Then
   the graph of addition is sequentially $\mc S$-recognizable.
\end{cor}

\noindent We now ready to prove feasibility of Ostrowski numeration systems.

\begin{cor}
Let $\alpha \in \R$ be a quadratic irrational number. Then $\mathcal O_\alpha$ is feasible and $\N_{\mc O_\alpha}=\N$.
\end{cor}

\begin{proof}

  Clearly $\N_{\mc O_\alpha}=\N$ and the order relation is sequentially $\mathcal O_{\alpha}$-recognizable by \cref{real_ostrowski_order}. Furthermore, each $1010 \cdots\star10\cdots$ is a normalized $\mc O_\alpha$-representation. Thus sequential regularity of normalization by \cref{addnormregular} follows from sequential $\mc O_{\alpha}$-recognizability of the graph of addition. This is \cite[Lemma 3.15]{H-Twosubgroups}.
\end{proof}

\begin{lem}\label{gammpowreg}
Let $\gamma$ be a Pisot number and let $n\in \N$. Then a set is sequentially (parallelly) $\mathcal{S}_{\gamma}$-recognizable if and only if it is sequentially (parallelly) $\mathcal{S}_{\gamma^n}$-recognizable.
\end{lem}
\begin{proof}
  Note that $\mc S_{\gamma^n}$-representations correspond to $\mc S_\gamma$-representations (in either case not necessarily normalized) which are zero at every index not divisible by $n$. Thus the claim follows from the regularity of normalization.
\end{proof}
\begin{lem}\label{gammultper}
  Let $\gamma\in\R_{>1}$ and let $w\in\bigcup_M\mc W_M$. If $w$ is ultimately periodic, then $[w]_{\mc S_\gamma}\in\Q(\gamma)$.
\end{lem}
\begin{proof}
  If $w=v0^\omega$ for $N\in\N,v\in\set{0,\dots,N,\star }^*$, this follows from $\gamma\in\Q(\gamma)$. Hence it remains to consider $w=v^\omega$. Then 
  \[
  [w]_{\mc S_\gamma}=[v0^\omega]_{\mc S_\gamma}\sum_{i=0}^\infty\gamma^{-i\abs v}=\frac{[v0^\omega]_{\mc S_\gamma}}{1-\gamma^{-\abs v}}\in\Q(\gamma).\qedhere
  \]
\end{proof}

\section{Recognizability and Definability}

Let $\mathcal{S}=(M,U,\rho)$ be a numeration system. In the following, we abuse notation and also use $U$ to denote the set $\{ U_i \ : \ i\in \Z\}$. For $i\in \N$, we define a binary relation $V_i$ on $\R$ such that for all $x,y\in \R$ 
\[
V_i(x,y) \text{ if and only if there is $j\in \Z$ such that $y=U_j$ and $i=\rho(x)_j$.}
\]
 Set  
\[
\mathcal{R}_{\mathcal{S}}:=(\R,{<},+,U_0,(V_i)_{i\in \N}), \quad
\mathcal{N}_{\mathcal{S}}:=(\N_{\mathcal{S}},{<},+,U_0,(V_i)_{i\in \N}).
\]
Let $\tau: U \to U$ be the map sending $U_i$ to $U_{-i}$ for each $i\in \Z$. We let $\mathcal{R}_{\mathcal{S}}^+$ be $(\mathcal{R}_{\mathcal{S}},\tau)$.

\begin{prop}\label{automatadef}
 The sets $U$ and $\N_{\mathcal{S}}$ are definable in $\mathcal{R}_{\mathcal{S}}$.
 \end{prop}
\begin{proof}
Check that $U$ is definable by $V_0(0,x)\lor V_1(0,x) \lor\dots\lor V_M(0,x)$.
We can define 
$\N_{\mathcal{S}}$ using 
\[\forall u\ (u\in U\land\abs u<\abs{U_0})\rightarrow V_0(x,u).\qedhere\]
\end{proof}

\noindent We will now establish two results connecting recognizability and definability. The arguments follow the proof of \cite[Theorem 6.1]{BHMV} using also ideas from the proof of \cite[Theorem 16]{BH-Betrand}.

\begin{thm}\label{thm:recdef}
  Suppose $\mathcal{S}$ is feasible. Let $X\subseteq \R_{\geq 0}^n$. Then
 \begin{enumerate}
  \item $X$ is sequentially $\mathcal{S}$-recognizable if and only if $X$ is definable in $\mathcal{R}_{\mathcal{S}}$.\label{feasseqreg}
  \item $X$ is parallelly $\mathcal{S}$-recognizable if and only if $X$ is definable in $\mathcal{R}_{\mathcal{S}}^+$.\label{feasparreg}
 \end{enumerate}
Moreover, if $X$ is bounded, then the conditions of (1) and (2) are equivalent.\\
For $X\subseteq\N_{\mathcal{S}}^n$ the following conditions are equivalent:
    \begin{enumerate}
    \item[(a)] $X$ is sequentially $\mathcal{S}$-recognizable.
    \item[(b)] $X$ is definable in $\mathcal{N}_\mathcal{S}$.
    \item[(c)] $X$ is parallelly $\mathcal{S}$-recognizable.
    \end{enumerate}
\end{thm}

\begin{proof}
    We first explain how definability implies recognizability. Since $\omega$-regular languages are closed under union, complementation and projection by \cref{fact:buechi}, it suffices to show that $+,<,U_0,V_i$ are sequentially $\mathcal{S}$-recognizable and $\tau$ is parallelly $\mathcal{S}$-recognizable. It is easy to see that $U_0,V_i,\tau$ are sequentially $\mathcal{S}$-recognizable. The order $<$ is sequentially $\mathcal{S}$-recognizable by feasibility of $\mathcal{S}$. Finally, addition is sequentially $\mathcal{S}$-recognizable by \cref{addregular}.\\
 
  \noindent We now consider (1). Let $X\subseteq \R_{\geq 0}^n$ be sequentially $\mathcal{S}$-recognizable. Let 
  \[\mathcal{A}=(Q,\Sigma_M^n \cup\{\star \},T,\set{p},F)\]
  be a Büchi automaton recognizing the sequential representation of $\rho(X)$. Since $\mathcal{S}$ is feasible, there is a $k\in \N$ such that every word of the form $a_n0^k\cdots a_00^k\star a_{-1}0^k\cdots$ is $\mathcal{S}$-normalized. \newline
  
  \noindent Let $w_1,\dots,w_n\in \mathcal{W}_M$ and let $N\in\N$ be such that $w_{i,j}=0$ for $i=1,\dots,n$ and $j\in \N_{>N}$. Set
  \[
 w:= ((w_{1,N},\dots,w_{n,N})\cdots (w_{1,0},\dots,w_{n,0})\star (w_{1,-1},\dots,w_{n,-1})\cdots). 
  \]
    Then the following are equivalent:
  \begin{itemize}
      \item $([w_1]_{\mc S},\dots,[w_n]_{\mc S})\in X$,
      \item  $\mathcal{A}$ accepts $w$,
    \item there is a sequence $(q_i)_{i\in \Z_{\leq N}}$ of states such that
  \begin{itemize}
  \item $q_{N}=p$,
  \item there is $q\in F$ such that $\{i \ : \ q_i=q\}$ is infinite,
  \item $(q_{i},(w_{1,i},\dots,w_{n,i}),q_{i-1}) \in T$ for every $i \in \Z_{\leq N}$.
  \end{itemize}
  \end{itemize}
  We encode the sequence $(q_i)_{-\infty<i\le N}$ by a $Q$-indexed tuple of extended $\omega$-words whose $q$-th entry is
 \begin{equation}\label{eq:theoremdefrec}\tag{$\ast$}
   \delta_{q,q_N}0^{k}\delta_{q,q_{N-1}}0^{k}\cdots\delta_{q,q_0}0^{k}\star \delta_{q,q_{-1}}0^{k}\cdots, 
 \end{equation} 
  where $\delta$ is the Kronecker-$\delta$. These words are $\mc S$-normalized by definition of $k$. Thus the existence of such a sequence of states can be expressed by the formula stating the existence of a $Q$-tuple $Z=(Z_q)_{q\in Q}\in\R^Q$ of real numbers encoding a valid sequence of states in this way and satisfying the above conditions. It remains to exhibit formulas for this.\newline

  \noindent The successor function $\sigma\colon U \to U$ mapping $U_i$ to $U_{i+1}$ is definable in $\mathcal{R}_{\mathcal{S}}$, since $|U_i|<|U_j|$ for all $i,j \in \Z$ with $i<j$. Clearly there are formulas $\phi(u,z)$, $\psi(u,z)$  and $(\psi_q(u,z))_{q\in Q}$ in the signature of $\mc R_{\mc S}$ such that for all $j\in \Z$, $q'\in Q$ and $Z=(Z_q)_{q\in Q}\in\R^Q$
\begin{itemize}
 \item $\mc R_{\mc S}\models \forall u \forall z (\phi(u,z)\lor\psi(u,z))\rightarrow u\in U$.
 \item $\mc R_{\mc S}\models \phi(U_j,Z)$ if and only if for all $q\in Q \ \rho(Z_q)_j = 0$
    \item $\mc R_{\mc S}\models \psi_{q'}(U_j,Z)$ if and only if for all $q\in Q$
   \[
   \rho(Z_{q})_j =    \begin{cases}
       1 & \text{ if } q=q',\\
       0 & \text{ otherwise}.
   \end{cases}
   \]
\end{itemize}
Let $\psi(u,z)$ be the formula $\bigvee_{q\in Q}\psi_q(u,z)$, and set
\begin{align*}
\chi_1(z)&:=\forall u\Big(\psi(u,z)\to\big(\phi(\sigma^{-1}u,z)\land\phi(\sigma^{-2}u,z)\land\dots\land\phi(\sigma^{-k}u,z)\land\psi(\sigma^{-k-1}u,z)\big)\Big),\\    
\chi_2(z)&:=\exists u\big(\psi(u,z)\land\forall u'(\abs {u'}>\abs u\land u'\in U\to\phi(u',z)\big),\\
\chi_3(z)&:=\psi(\sigma^{-1}U_0,z).
\end{align*}
The reader can now easily check that for every $Z=(Z_q)_{q\in Q}\in\R^Q$ the following are equivalent:
\begin{itemize}
    \item $\mc R_{\mc S}\models\chi_1(Z)\land\chi_2(Z)\land\chi_3(Z)$,
    \item there is a sequence $(q_i)_{i\in \Z_{\leq N}}$ such that $Z_q$ is of the form \eqref{eq:theoremdefrec} for every  $q\in Q$.
\end{itemize}
 Using $V_i$ and $\delta$, one can also easily constructs a formula $\eta_{r,s}(u,v_1,\dots,v_n)$ in the signature of $\mc R_{\mc S}$ such that 
 \[
 \mc R_{\mc S}\models\eta_{r,s}(U_j,[w_1]_{\mc S},\dots,[w_n]_{\mc S}) \text{ if and only if } (r,(w_{1,j},\dots,w_{n,j}),s)\in T.
 \]
 Using these observations one can verify that $\mc A$ accepts $w$ if and only if
  \begin{align*}
    \mc R_{\mc S}\models&\exists z(\chi_1(z)\land\chi_2(z)\land\chi_3(z)\\
    \land&\exists u\big(\psi_p(u,z)\land\forall u'(\abs {u'}>\abs u\land u'\in U\to\phi(u',z)\big)\\
    \land&\bigvee_{q\in F}\forall \eps>0\exists u(\abs u<\eps\land\psi_q(u,z))\\
    \land&\bigwedge_{r,s\in Q}\forall u(\psi_r(u,z)\land\psi_s(\sigma^{-k-1}u,z)\to\eta_{r,s}(u,[w_1]_{\mc S},\dots,[w_n]_{\mc S})).
    \end{align*}
This completes the proof of (1).\newline

\noindent  For (2) we only need to modify the construction for (1): We have $N=-1$, which will only simplify the formula. However, for checking the transition relation, we need not only access to $V_i$ of $U_j$ and the input, but also of $V_i$ of $U_{-j}$ and the input. This can be achieved using $\tau$. We leave the details for the reader to check.\newline
  
    \noindent To show that (1) implies (2) for bounded $X$, we may reduce to the case where the integer part of each member of $X$ is zero. Then the graph of the map sending the sequential representation to the parallel representation is clearly regular.\newline
  
  \noindent We now show the equivalence of (a), (b) and (c). We already know (b)$\implies$(a)$\implies$(c). It remains to show that parallelly $\mathcal{S}$-recognizable subsets of $\N^n_{\mathcal{S}}$ are $\mc N_{\mc S}$-definable. This can be done by the same argument as for (2). One only needs to observe that $V_i(\tau u,x)$ is constantly $0$, when $x\in \N_{\mc S}$ and $u \in \set{ U_i : i > 0}$. Thus in the formula in (2) we can replace such occurrences of $V_i$ by $0$.
\end{proof}

\begin{cor}
Suppose $\mc S$ is feasible. Then the first-order theory of $\mc R^+_{\mc S}$ is decidable.
\end{cor}
\begin{proof}
It suffices to show the decidability of the first-order theory of the substructure with underlying set $\R_{\ge0}$. Note that the construction in \cref{thm:recdef}(2) is effective. Thus the decidability of the first-order theory can be reduced to the emptiness problem for Büchi automata which is well-known to be decidable. 
\end{proof}

\begin{thm}\label{alphgammreg}
  Let $\alpha$ be a quadratic irrational, let $\gamma\in\mc O_\alpha^\times$ with $\gamma>1$, and let $X\subseteq\R_{>0}^n$. The following statements are equivalent:
  \begin{enumerate}
  \item $X$ is parallelly $\mathcal{O}_{\alpha}$-recognizable,
  \item $X$ is parallelly $\mathcal{S}_{\gamma}$-recognizable,
  \item $X$ is definable in $(\R,{<},+,\Z,\alpha\Z)$,
  \item $X$ is definable in $(\R,{<},+,\Z,x\mapsto\alpha x)$.
  \end{enumerate}
\end{thm}
\begin{proof}

  The implication (3)$\implies$(4) is immediate. For (1)$\implies$(3), by \cref{thm:recdef} it suffices to show that $(\R,{<},+,\Z,\alpha\Z)$ defines $\mc R_{\mc O_{\alpha}}^+$. This is essentially \cite[Section 4]{H-Twosubgroups}.\newline
  
 \noindent Consider the equivalence (1)$\iff$(2). It suffices to show that the bijection 
 \[
 \Phi : \rho^{\alpha}(\R_{\geq 0}) \to \rho_{\gamma}(\R_{\geq 0})
 \] 
 that for every $x \in \R_{\geq 0}$ maps the normalized $\mc O_\alpha$-representation $\rho^{\alpha}(x)$ to the normalized $\mc S_\gamma$-representation $\rho_{\gamma}(x)$, is parallelly regular. Indeed, then its inverse is also parallelly regular, and by composing with these maps we can show every subset of $\R^n_{\geq 0}$ that is parallelly $\mc O_{\alpha}$-recognizable, is also parallelly $\mc S_{\gamma}$-recognizable and vice versa.\newline

 \noindent By Dirichlet's unit theorem, $\mc O_\alpha^\times/\set{-1,+1}$ is a free abelian group of rank $1$. Thus there are $k,\ell\in \N$ such that $\gamma^k=\norm{\Gamma_\alpha}^\ell$. Hence, by \cref{gammpowreg} we may assume $\gamma=\norm{\Gamma_\alpha}$ and $\det \Gamma_\alpha=1$.
 Let $m,C,D,E$ be given by \cref{shifting}.
 \newline

  \noindent Let $N\in\N$. We first construct $M\in \N$ and a parallelly regular map $f_N\colon \mc W_{M^{\alpha}}\to\mc W_{M}$ such that for all $w\in \mc W_{M^{\alpha}}$
  \begin{enumerate}
      \item[(a)] $[f_N(w)]_{\mc O_\alpha}=[w]_{\mc O_\alpha}$, and 
      \item[(b)] all entries of $f_N(w)$ at indices $i\in \Z$ with $|i|<N$ vanish.
  \end{enumerate}
  It is easy to see that there is an $M\in\N$ such that every nonnegative real number has a (not necessarily normalized) $\mc O_\alpha$-representation all whose entries at indices $i\in \Z$ with $|i|<N$ vanish. Then for $w\in\mc W_{M^{\alpha}}$, we define $f_N(w)\in \mc W_M$ as the lexicographically maximal $\mc O_\alpha$-representation of $[w]_{\mc O_\alpha}$ in $\mc W_{M}$ such that (b) holds. By parallell regularity of the lexicographic order and $\mc O_{\alpha}$-normalization, the function $f_N$ is parallelly regular. Set $f:=f_{mP(\alpha)}$.\newline

\noindent For each $k \in \{1,\dots,P(\alpha)\}$, let $V_k$ be the set of all $w=(w_i)_{i \in \Z} \in\mc W_{M}$ such that for all $i\in \Z$
  \begin{itemize}
      \item  $w_i=0$ if $i\not\equiv k\mod P(\alpha)$, and
      \item $w_i=0$ if $|i|<m P(\alpha)$.
  \end{itemize}
Let $V_k^{\Z}$ be the subset of all $w=(w_i)_{i \in \Z} \in V_k$ such that 
$w_i=0$ for all $i \in \Z_{<0}$, and let $V_k^{\R}$ be the subset of all $w=(w_i)_{i \in \Z} \in V_k$ such that 
$w_i=0$ for all $i \in \Z_{\geq 0}$. These sets are easily seen to be parallelly regular. Futhermore, there is a parallelly regular map $g = (g_1,\dots,g_{2P(\alpha)}) :\mc W_M \to \mc W_M^{2P(\alpha)}$ such that $(g_{2k-1}(w),g_{2k}(w)) \in V_{k}^{\Z} \times V_k^{\R}$ for all $k\in \{1,\dots,P(\alpha)\}$, and $[w]_{\mc S_\gamma} = \sum_{\ell=1}^{2P(\alpha)} [g_{\ell}(w)]_{\mc S_\gamma}$.\newline

\noindent Let $w\in V_{k}^{\Z}$. Then $w$ is of the form
  \begin{equation}\label{eq:gammaalpha1}\tag{$\ast$}
w_n0^{P(\alpha)-1}\cdots 0^{P(\alpha)-1}w_20^{P(\alpha)-1}w_10^{k-1+mP(\alpha)}\star 0^\omega.
  \end{equation}
Let $h_k^{\Z,1} : V_{k}^{\Z} \to \mc W_{M}$ be the function that maps $w$ of the form \eqref{eq:gammaalpha1} to 
  \[
  w_n\cdots w_2w_1\star 0^\omega,
\]
and let $h_k^{\Z,2} : V_{k}^{\Z} \to \mc W_{M}$ be the function that maps $w$ of the form \eqref{eq:gammaalpha1} to 
  \[
  0\star w_1w_2\cdots w_n0^\omega.
\]
Since reversing and shifting are parallelly regular, both functions are parallelly regular. We obtain from \cref{shifting} that
  \[
  [w]_{\mc O_\alpha} = C[w_n\dots w_2w_1\star 0^\omega]_{\mc S_\gamma}+D[0\star w_1w_2\dots w_n0^\omega]_{\mc S_\gamma}
  = C [h_k^{\Z,1}(w)]_{\mc S_\gamma} + D [h_k^{\Z,2}(w)]_{\mc S_\gamma}.
  \]
Let $w\in V_{k}^{\R}$. Then $w$ is of the form
  \begin{equation}\label{eq:gammaalpha2}\tag{$\ast\ast$}
 0\star 0^{k-1+mP(\alpha)}w_10^{P(\alpha)-1}w_20^{P(\alpha)-1}\cdots.
  \end{equation}
 Let $h_k^{\R} : V_{k}^{\R} \to \mc W_{M}$ be the function that maps $w$ of the form \eqref{eq:gammaalpha2} to  $0\star w_1w_2\cdots$. Since shifting is parallelly regular, this function is also parallelly regular. By \cref{shifting}
 \[
   [w]_{\mc O_\alpha} = E[0\star w_1w_2\dots]_{\mc S_\gamma} = [h_k^{\R}(w)]_{\mc S_\gamma}. 
  \]

\vspace{0.2cm}

\noindent Let $w\in \rho^{\alpha}(\R_{\geq 0})$. Combing the above observations, we obtain
\[
[w]_{\mc O_{\alpha}} = \sum_{k=1}^{P(\alpha)} \Big( C [(h_k^{\Z,1}\circ g_{2k-1} \circ f)(w)]_{\mc S_\gamma} + D [(h_k^{\Z,2}\circ g_{2k-1} \circ f)(w)]_{\mc S_\gamma}\Big) + \sum_{k=1}^{P(\alpha)}  E [(h_k^{\R}\circ g_{2k} \circ f)(w)]_{\mc S_\gamma}.
\]
Thus
\[
\Phi(w) = \rho_{\gamma}\Bigg(\sum_{k=1}^{P(\alpha)} \Big( C [(h_k^{\Z,1}\circ g_{2k-1} \circ f)(w)]_{\mc S_\gamma} + D [(h_k^{\Z,2}\circ g_{2k-1} \circ f)(w)]_{\mc S_\gamma}\Big) + \sum_{k=1}^{P(\alpha)}  E [(h_k^{\R}\circ g_{2k} \circ f)(w)]_{\mc S_\gamma}\Bigg).
\]
Since $\mc S_{\gamma}$-normalization is parallelly regular and scalar multiplication and addition are parallelly $\mc S_\gamma$-recognizable, it follows that $\Phi$ is parallelly regular.\newline 
  
  \noindent For the implication (4)$\implies$(1), it suffices to show that $\Z,+,<,\gamma\cdot-$ are parallelly $\mc O_\alpha$-recognizable, since given $+$, multiplication with $\alpha$ and $\gamma$ are interdefinable. By \cref{thm:recdef} it is enough to show that these sets are definable in $\mc R_{\mc O_{\alpha}}^+$. For addition and order this is is immediate, and for $\Z$ it follows from \cref{automatadef}. Note that multiplication by $\gamma$ corresponds to a shift in the $\mc S_\gamma$-representation and hence is definable in $\mc R_{\mc S_{\gamma}}^+$. Thus multiplication by $\gamma$ is definable in $\mc R_{\mc O_{\alpha}}^+$ by the equivalence of (1) and (2) and \cref{thm:recdef}.
\end{proof}

\section{The main argument}

In this section we present the main technical step in the proof of Theorem B. The thrust of the argument is essentially the same as in \cite[Section 5\&6]{BBB10}, exploiting product-stability in different numeration systems. As before, we need to make non-trivial adjustments for our use of irrational bases. We begin with a statement that is established for $r$- and $s$-recognizable subsets of $\R$ in \cite[Section 5.1]{BBB10}.
\begin{lem}\label{badpointsex}
    Let $\alpha_1,\dots,\alpha_n$ be irrational Pisot numbers with $\bigcap_{i=1}^n\Q(\alpha_i)=\Q$, and let $X\subseteq[0,1]^d$ be such that $X$ is sequentially $\mc S_{\alpha_i}$-recognizable for $i\in\set{1,\dots,n}$, but not definable in $(\R,{<},+,1)$. 
  Then there is a bad point with respect to $X$ that lies in $\Q^d$.
\end{lem}

\begin{proof}
  Let $B$ be the set of points $x\in[0,1]^d$ such that there is no open box $U$ around $x$ such that $U\cap X$ is a finite union of polyhedra. Note that $B$ is a subset of the set of bad points with respect to $X$. Furthermore, $B$ is definable in $(\R,{<},+,1,X)$ by \cref{selfdef}, thus by \cref{thm:recdef} sequentially $\mc S_{\alpha_i}$-recognizable for $i\in\set{1,\dots,n}$.\newline
  
  \noindent Towards a contradiction, assume $B=\emptyset$. Then for each $x \in [0,1]^d$ there is an open box $U_x$ around $x$ such that $U_x\cap X$ is a finite union of polyhedra. By compactness finitely many $U_x$ cover $[0,1]^d$ and so $X$ is a finite union of polyhedra. By \cref{verticesdef} there is a finite set $V$ definable in $(\R,{<},+,1,X)$ containing all vertices of these polyhedra. Since $V$ is finite and is sequentially $\mc S_{\alpha_i}$-recognizable, then it follows from \cref{gammultper} that
  \[
  V\subseteq\bigcap_{i=1}^n\Q(\alpha_i)^d=\Q^d.
  \]
   Thus $X$ is definable in $(\R,{<},+,1)$, contradicting the assumptions on $X$.\newline 
   
  \noindent Hence $B$ is nonempty. Since $B$ is compact, there is a lexicographically minimal $p\in B$. It remains to show $p\in\Q^d$. Since $\{p\}$ is definable in $(\R,{<},+,1,B)$, it is sequentially $\mc S_{\alpha_i}$-recognizable for all $i\in\set{1,\dots,n}$. As above, it follows that $p \in \Q^d$.
  \end{proof}

\subsection{Product and sum stability} We now recall the definitions of product- and sum-stability.

\begin{defn}
Let $X,D\subseteq \R$ and let $t \in \R$. We say $X$ is 
    \textbf{$t$-product stable in $D$} if for all $x\in D$ with $tx\in D$
    \[
    x \in X \text{ if and only if } tx \in X.
    \]
    We say $X$ is \textbf{$t$-sum stable in $D$} if for all $x\in D$ with $x+t\in D$
    \[
    x \in X \text{ if and only if } x+t \in X.
    \]
For $X\subseteq\R_{> 0}$ we set
\begin{align*}
\Pi(X)&:=\{t \in \R_{>0} \ : \ X \text{ is $t$-product stable in } \R_{>0}\}, \\
\Sigma(X) &:= \{ t \in \R \ : \ X \text{ is $t$-sum stable in } \R_{>0}\}.
\end{align*}

\end{defn}

\noindent If $X$ is parallelly $\mc S_\alpha$-recognizable for a Pisot number $\alpha$, then $\Sigma(X)$ is parallelly $\mc S_\alpha$-recognizable, as it is defined by a formula. Moreover, $\Sigma (X)$ is a subgroup of $(\R,+)$.\newline

\noindent We now prove the following analogue of \cite[Lemma 5.2]{BBB10}.  
\begin{lem}\label{lem:prodstable}
    Let $\alpha_1,\dots,\alpha_n$ be irrational Pisot numbers with $\bigcap_{i=1}^n\Q(\alpha_i)=\Q$, and let $X\subseteq[0,1]^d$ be parallelly $\mc S_{\alpha_i}$-recognizable for all $i\in\set{1,\dots,n}$, but not definable in $(\R,{<},+,1)$.  Then there is $Y\subseteq[0,1]^d$ and $N\in \N_{>0}$ such that
  \begin{enumerate}
    \item $Y$ is sequentially $\mathcal{S}_{\alpha_i}$-recognizable for all $i\in\set{1,\dots,n}$,
    \item $Y$ is $\alpha_i^N$-product stable in $[0,1]^d$ for all $i\in \{1,\dots,n\}$,
    \item $Y$ is not definable in $(\R,{<},+,1)$.
  \end{enumerate}
\end{lem}
\begin{proof}
  By \cref{badpointsex,zerobad} we may assume that $0$ is a bad point for $X$. Let $i\in \{1,\dots,n\}$ and let $\mc A_i$ be a total deterministic Muller automaton sequentially recognizing $\rho_{\alpha_i}(X)$ over the alphabet $\Sigma_i=\set{0,\dots,\floor{\alpha_i}}^d\cup \set{\star }$. Hence there is a path of length $m_i$ from the initial state to a cycle of length $N_i$, with all labels $0\in\N^d$. Thus for every  word $w\in\Sigma_i^*$, 
  \[
  \mc A_i \text{ accepts } \star 0^{m_i}w \text{ if and only if } \mc A_i \text{ accepts } \star 0^{m_i+N_i}w.
  \]
  However, 
  \[
  [0^{m_i+N_i}w]_{\mc S_{\alpha_i}}=\alpha_i^{N_i}[0^{m_i}w]_{\mc S_{\alpha_i}}.
  \]
  Hence $X\cap[0,\alpha^{-m_i}]$ is $\alpha_i^{N_i}$-product stable in $[0,\alpha_i^{-m_i}]$.\newline

  \noindent Now let $N$ be the least common multiple of $N_1,\dots,N_d$, and let $q\in \Q$ be larger than $\alpha_i^N$ for each $i\in \{1,\dots,d\}$. Set $Y=qX\cap[0,1]$. Then $Y$ is $\alpha_i^N$-product stable in $[0,1]^d$ and sequentially $\mc S_{\alpha_i}$-recognizable for $i=1,\dots,n$. Since $0$ is a bad point with respect to $X$, we know that $Y$ is not definable in $(\R,{<},+,1)$.
\end{proof}

\noindent Note that \cite[Lemma 5.2]{BBB10} only handles the case $d=1$. We now show the stronger statement that $Y$ can be taken to be a subset of $[0,1]$. The proof is based on ideas from \cite[Section 3.2]{BBL09}.

\begin{prop}\label{definmultidim}
  Let $\alpha,\beta$ be irrational Pisot numbers with $\Q(\alpha)\cap\Q(\beta)=\Q$, and let $X\subseteq[0,1]^d$ be parallelly $\mc S_\alpha$-recognizable and $\mc S_\beta$-recognizable, but not definable in $(\R,{<},+,1)$. Then there is an $\mc S_\alpha$-recognizable subset of $[0,1]$ that is not definable in $(\R,{<},+,1)$, and $\alpha^N$-product stable and $\beta^N$-product stable in $[0,1]$ for some $N\in\N_{>0}$.
\end{prop}
\begin{proof}
  By \cref{lem:prodstable} we may assume that $X$ is $\alpha^N$- and $\beta^N$-product stable in $[0,1]^d$ for some $N>0$. We prove the statement by induction on $d$. The case $d=1$ is trivial. Now let $d> 1$. 
  By the induction hypothesis we may assume that $X\cap\bd([0,1]^d)$ is definable in $(\R,{<},+,1)$, and hence a finite union of polyhedra with rational vertices. We now show that either the conclusion of the proposition holds or
  \begin{equation}\label{eq:definmultidim}
  X=\set{tx\ : \ t\in\R\text{ and }0<t\le1\text{ and }x\in X\cap\bd([0,1]^d)}.    
  \end{equation}
  Since \eqref{eq:definmultidim} yields a contradiction, this is enough to finish the whole proof.\newline

  \noindent By \cref{regtestultper} we just need to show that \eqref{eq:definmultidim} holds after intersecting both sides with $\Q(\alpha)^d$. For this, it suffices to show that for all $z\in\Q(\alpha)^d\cap\bd([0,1]^d\setminus\set0)$, 
  \[
\text{either } \{ t z \ : t \in (0,1] \} \subseteq X \text{ or  } \{ tz : t\in(0,1]) \} \cap X =\emptyset.
  \]
  Let $z\in\Q(\alpha)^d\cap\bd([0,1]^d\setminus\set0)$. The linear map $f\colon\R \to \R^d$ sending $1$ to $z$ is definable in $(\R,{<},+,1,X,x\mapsto \alpha x)$ since $z\in \Q(\alpha)^d$, and thus sequentially $\mc S_\alpha$-recognizable. Set $Z := f^{-1}(X)$.\\
  Note that $Z$ is $\alpha^N$-product stable and $\beta^N$-product stable in $[0,1]$ because $X$ is so. Thus if $\{tz \ : \ t\in(0,1]\}\cap X$ is not definable in $(\R,{<},+,1)$, then $Z$ satisfies the conclusion of the proposition. Hence we may assume that $Z$ is definable in $(\R,{<},+,1)$, and thus is a finite union of intervals and points. We have to show that $Z$ is empty or contains $\lointer{0,1}$.\\
  Replacing $\alpha$ by $\alpha^{-1}$ if necessary, assume $\alpha>1$. Suppose $x\in Z$. Then $\alpha^{-kN}x\to 0$ as $k\to \infty$. Since $Z$ contains $\alpha^{-kN}x$ for $k\in \N$ and is a finite union of intervals and points, it contains an interval $(0,\eps)$. Then $Z$ contains $(0,\alpha^{kN}\eps)\cap[0,1]$ for all $k\in\N$ and thus $\lointer{0,1}\subseteq Z$.
\end{proof}

\subsection{Ultimate periodicity}

In this subsection, we show that $\mc{S}_{\alpha}$-recognizable sets that are both $\alpha$- and $\beta$-product stable in $\R_{>0}$, are eventually $p$-sum stable for some $p$. This roughly corresponds to the reduction in \cite[Lemma 6.3]{CLR}, although we again have to use different arguments. In particular, we borrow some ideas and notation from Krebs \cite{Krebs-Reasonable}.

\begin{defn}
Let $X,Y\subseteq\R$ be such that $Y\subseteq X$. We say $X$ has \textbf{local period $p$} on $Y$ if $X\cap Y$ is $p$-sum stable on $Y$. We say $X$ is \textbf{ultimately periodic} with period $p$ if there is $a\in \R$ such that $X$ has local period $p$ on $[a,\infty)$.
 \end{defn} 

\noindent The main result we prove in this subsection is the following.  
\begin{prop}\label{prodstabultper}
 Let $\alpha$ be a Pisot number, let $X \subseteq \R_{\geq 0}$ be sequentially $\mathcal{S}_{\alpha}$-recognizable such that $\Pi(X)$ is dense in $\R_{>0}$. Then $X$ is ultimately periodic.
\end{prop}

\noindent Before we give the proof of \cref{prodstabultper}, we need two lemmas.

\begin{lem}\label{glueperint}
  Let $X\subseteq \R$, and let $a,b,c,d,p,q\in \R$ be such that 
  \begin{enumerate}
      \item $a<b<c<d$ and $c-b\ge p+q$,
      \item $X$ has local period $p$ on $[a,c]$, and
      \item $X$ has local period $q$ on $[b,d]$.
  \end{enumerate}
  Then $X$ has local period $p$ on $[a,d]$. The same statement holds for open intervals instead of closed intervals.
\end{lem}
\begin{proof}
Let $x\in [a,d]$ be such that 
$x+p\in[a,d]$. We want to show that $x\in X$ if and only if $x+p\in X$.
If $x+p\le c$, then this follows immediately, since $X$ has local period $p$ on $[a,c]$. Now suppose $x+p>c$. Since $c-b\ge p+q$, there is $n\in\N$ such that
\[
x-nq\in[b,b+q]\subseteq[b,c-p].
\]
Then $x-nq,x-nq+p\in[a,c]$, and thus $x-nq\in X$ if and only if $x-nq+p\in X$. 
Since $X$ has local period $q$ on $[b,d]$, we have
  \[x\in X\iff x-nq\in X\iff x-nq+p\in X\iff x+p\in X.\qedhere\]
\end{proof}

\begin{lem}\label{ultperlocally}
Let $X \subseteq \R_{> 0}$, let $a,b,p\in \R_{>0}$ be such that 
\begin{enumerate}
    \item $a<b$, $2p<b-a$,
    \item $X$ has local period $p$ on $(a,b)$, and
    \item $\Pi(X)$ is dense in $\R_{>0}$.
\end{enumerate}
Then $X$ is ultimately periodic.
\end{lem}
\begin{proof}
  Let $B$ be the set of $c\in\R_{>a}$ such that $X$ has local period $p$ on $(a,c)$. It suffices to show $B$ unbounded. Towards a contradiction, assume that $B$ is bounded. Set $d:=\sup B$. Since $d\geq b$, we have $2p<d-a$. Let $\eps>0$ be such that $2p+\eps p<d-(1+\eps)a$ and $X$ is $1+\eps$-product stable. Since $X=(1+\eps)X$, we know that $X$ has local period $(1+\eps)p$ on the interval $\big((1+\eps)a,(1+\eps)d\big)$. Thus Lemma \ref{glueperint} yields that $X$ has local period $p$ on $\big(a,(1+\eps)d\big)$. This contradicts $d=\sup B$.
\end{proof}

\begin{proof}[Proof of \cref{prodstabultper}] 
Since $\mc S_{\alpha}$-normalization is sequentially regular, there is a deterministic Muller automaton $\mathcal{A}=(Q,(\Sigma_{\ceil\alpha}\cup \{\star \}),T,I,F)$ accepting a word if and only if it is an $\mc S_\alpha$-representation of an element of $X$. For each state $q\in Q$, let $A_q$ be the set of $n\in\N$ such that the automaton is in state $q$ after reading $10^n\star$.\\
 By van der Waerden's theorem on arithmetic progressions there are $q\in Q$ and $k,m \in \N$ such that $m,m+k,m+2k$ lie in $A_q$. Every real number in $[\alpha^\ell,\alpha^\ell+1]$ has an $\mc S_{\alpha}$-representation starting with $10^\ell\star$. Thus for $z\in[0,1]$, we have 
 \begin{equation}\label{eq:prodstabultper0}
 \alpha^m+z\in X \iff \alpha^{m+k}+z\in X \iff \alpha^{m+2k}+z\in X.    
 \end{equation}
  As $\Pi(X)$ is dense, there is $t\in \R$ such that $X$ is $t$-product stable and 
  \[0 < t-\alpha^k < \frac{1}{2\alpha^{m+k}}.\]
 For ease of notation, set $\eps:=(t-\alpha^k)\alpha^{m+k}$ and $\eps':=(t-\alpha^k)\alpha^m$. Since $0<\eps'<\eps<\frac12$, we have
\begin{equation}\label{eq:prodstabultper1}
 0<\eps-\eps'<\frac{1-\eps'}2.
\end{equation}
For $z\in\frac1t\inter{0,1-\eps}$, we obtain by \eqref{eq:prodstabultper0} and $t$-product stability of $X$ that
  \begin{align*}
    \alpha^{m+k}+\eps+tz\in X&\iff\alpha^{m+2k}+\eps+tz\in X\iff t(\alpha^{m+k}+z)\in X\\
    &\iff\alpha^{m+k}+z\in X\iff\alpha^m+z\in X\\
    &\iff t(\alpha^m+z)\in X\iff\alpha^{m+k}+\eps'+tz\in X.
  \end{align*}
  Thus $X$ has local period $\eps-\eps'$ on $\inter{\alpha^{m+k}+\eps',\alpha^{m+k}+1}$. With \eqref{eq:prodstabultper1} \cref{ultperlocally} finishes the proof.
\end{proof}

\subsection{Regular product stable}
Let $\alpha,\beta$ be multiplicatively independent Pisot numbers such that $\beta\notin\Q(\alpha)$. The goal of this subsection is to prove the following proposition.

\begin{prop}
\label{prodstabletriv}
Let $X\subseteq\R_{> 0}$ be sequentially $\mathcal{S}_{\alpha}$-recognizable and $\alpha$- and $\beta$-product stable in $\R_{>0}$. Then $X$ is either $\emptyset$ or $\R_{>0}$.
\end{prop}

\noindent Our proof is based on the argument given \cite[Section 6.2]{BBB10}.  We need the following lemma first. 

\begin{lem}\label{sumstable}
Let $X\subseteq\R_{>0}$ be sequentially $\mathcal{S}_{\alpha}$-recognizable and $\alpha$- and $\beta$-product stable in $\R_{>0}$. 
  Then there is $Y\subseteq\R_{>0}$ such that
  \begin{enumerate}
  \item $Y$ is sequentially $\mathcal{S}_{\alpha}$-recognizable,
  \item If $Y=\emptyset$, then $X=\emptyset$, and if $Y=\R_{>0}$, then $X=\R_{>0}$,
  \item $Y$ is $\alpha$- and $\beta$-product stable in $\R_{>0}$,
  \item $\alpha^k,\beta^k\in\Sigma(Y)$ for every $k\in\Z$, and
  \item there are $m\in\N$ and $x\in\Q(\alpha)^\times\cap\Sigma(Y)$ such that $\frac x{\alpha^{mn}-1}\in\Sigma(Y)$ for all $n\in\N_{>0}$.
 \end{enumerate}
 \end{lem}
   
 \begin{proof}
  Note that $X$ is product stable for the dense set $\alpha^{\Z}\beta^{\Z}$. By \cref{prodstabultper} we know that $X$ is ultimately periodic. Let $x\in \R_{\ge0}$ and $p\in \R_{>0}$ be such that for all $y\in \R_{\ge x}$
  \[y\in X\iff y+p\in X.\]
  Replacing $x$ by a larger number, we may assume $x\in\Q(\alpha)$. Thus $\Sigma((X-x)\cap\R_{\geq 0})$ is sequentially $\mc S_\alpha$-recognizable and nonempty, and by \cref{regtestultper,gammultper} it contains some $q\in\Q(\alpha)$. Set $Y:=\frac1qX$. It is easy to see that $Y$ satisfies (1)-(3) and has local period $1$ on $[\frac xq,\infty)$. Let $k\in \N$. By (3) the set $Y$ also has local period $\alpha^{-k}$ on $[\frac x{\alpha^kq},\infty)$, and thus using \cref{glueperint} we see that it has local period $1$ on $\rointer{\frac x{\alpha^kq},\infty}$. Since $k$ is arbitrary, we get that $1\in\Sigma(Y)$. Now (4) follows from (3).\newline

\noindent 
  We now establish (5). Since $\Sigma (Y)\cap[0,1]$ is sequentially $\mathcal{S}_{\alpha}$-recognizable, we obtain by \cref{condomegreg} nonempty regular languages $K_1,\dots,K_m,L_1,\dots,L_m \subseteq \{0,\dots,\floor{\alpha}\}^*$ such that $L_i$ does not contain the empty word, $L_iL_i=L_i$ for each $i\in \{1,\dots,m\}$, and
  \[\Sigma (Y)\cap[0,1]=\bigcup_{i=1}^m\set{[wv_1v_2\cdots]_{\mc S_{\alpha}} \ : \ w\in K_i,v_1,v_2,\dots \in L_i}.\]
There is $i\in \{1,\dots, m\}$ such that $L_i$ contains at least two elements of the same length. Indeed, otherwise all words in $\bigcup_{i=1}^mK_iL_i^\omega$ are ultimately periodic, and hence by \cref{gammultper} 
\[\Sigma(Y)\cap[0,1]\subseteq\Q(\alpha),\]
contradicting $\beta^{-1}\in\Sigma(Y)\cap[0,1]$.\newline

\noindent  Now fix $i\in\{1,\dots,m\}$ such that $L_i$ contains two distinct elements $v_1,v_2$ of the same length, and fix $w\in K_i$. Then for all $n\in\N$
\[
[w(v_1^nv_2)^\omega]_{\mc S_\alpha} \in \Sigma (Y)\cap[0,1].
\]
Note that every number with terminating $\mc S_{\alpha}$-representation lies in $\Sigma(Y)$, and even in $\Pi(\Sigma(Y))$. Applying this to $[w0^\omega]_{\mc S_\alpha}$, we obtain $[(v_1^nv_2)^\omega]_{\mc S_\alpha}=[w(v_1^nv_2)^\omega]_{\mc S_\alpha}-[w0^\omega]_{\mc S_\alpha}\in\Sigma(Y)$ for all $n\in\N$. And replacing $\alpha$ by $\alpha^{\abs{v_1}}$ we may assume $\abs{v_1}=\abs{v_2}=1$. 
  Now set
 \[
 x:=\Big([v_20^\omega]_{\mc S_\alpha}\alpha-[v_10^\omega]_{\mc S_\alpha}\alpha\Big)(\alpha-1).
 \]
Since $x$ has a terminating $\mc S_{\alpha}$-representation, we have that $x\in\Sigma(Y)$.
Then   \begin{align*}
  \frac x{\alpha^{n+1}-1}+[v_10^\omega]_{\mc S_\alpha}\alpha &= \Big([v_20^\omega]_{\mc S_\alpha}\alpha-[v_10^\omega]_{\mc S_\alpha}\alpha\Big)\frac{\alpha-1}{\alpha^{n+1}-1}+\frac{[v_10^\omega]_{\mc S_\alpha}\alpha^{2}(\alpha^{n}-1+\alpha^{-1}(\alpha-1))}{\alpha^{n+1}-1}  \\
    &=(\alpha-1)\frac{[v_20^\omega]_{\mc S_\alpha}\alpha+[v_10^\omega]_{\mc S_\alpha}\alpha^2\frac{\alpha^n-1}{\alpha-1}}{\alpha^{n+1}-1} \\
    &=(\alpha-1)\alpha^{n+1}\frac{[v_20^\omega]_{\mc S_\alpha}\alpha^{-n}+\sum_{i=0}^{n-1}[v_10^\omega]_{\mc S_\alpha}\alpha^{-i}}{\alpha^{n+1}-1} \\
    &=(\alpha-1)\alpha^{n+1}\frac{[v_1^nv_20^\omega]_{\mc S_\alpha}}{\alpha^{n+1}-1}\\
    &=(\alpha-1)\sum_{i=0}^\infty[v_1^nv_20^\omega]_{\mc S_\alpha}\alpha^{-i(n+1)} \\
    &=(\alpha-1)[(v_1^nv_2)^\omega]_{\mc S_\alpha}\in\Sigma(Y)
  \end{align*}
and so $\frac x{\alpha^{n+1}-1}\in\Sigma(Y)$.
\end{proof}

\begin{proof}[Proof of Proposition \ref{prodstabletriv}]
Let $Y \subseteq \R_{> 0}$ be as given by \cref{sumstable}. Replacing $\alpha$ by $\alpha^m$, let $x\in\Q(\alpha)^\times\cap\Sigma (Y)$ be such that $\frac x{\alpha^n-1}\in\Sigma (Y)$ for all $n\in \N_{>0}$. Set $Z:=\frac1xY$.  Note that $\Sigma(Z)$ is an $\alpha$-and $\beta$-product stable subgroup and $\Sigma(Z)=\frac1x\Sigma(Y)$. In particular, $1\in\Sigma(Z)$. Hence $\Sigma(Z)\cap[0,1]$ contains all numbers with terminating $\mc S_{\alpha}$-representation, and is product stable for them.\newline

\noindent Now we show that all numbers with purely periodic $\mc S_{\alpha}$-representation are in $\Sigma(Z)$. Indeed, for $v \in \{0,\dots,\floor{\alpha}\}^*$ we have that
  \[[v^\omega]_{\mc S_\alpha}=\sum_{i\ge1}[v0^\omega]_{\mc S_\alpha}\alpha^{-i\abs v}=\frac1{\alpha^{\abs v}-1}\cdot[v0^\omega]_{\mc S_\alpha}\in\Sigma (Z),\]
  because $\frac1{\alpha^n-1}\in\Sigma (Z)$ for all $n\in \N_{>0}$.\\
  All numbers with ultimately periodic $\mc S_{\alpha}$-representation lie in $\Sigma(Z)$, as they may be written as sums of such numbers and ones with terminating representation. By \cref{regtestultper} we conclude that $\Sigma(Z)=\R$ and the claim follows.
\end{proof}

\section{Proof of the main theorems}

In this section we finish the proof of Theorem A and B. By \cref{definbounded} we just need to handle the cases of subsets of $\N^n$ and subsets of $[0,1]^n$. 

\subsection{Subsets of $\N^n$}
We first consider the case of a subset of $\N^n$. Here we prove the following analogue Bès' Cobham-Semënov theorem for linear numeration systems \cite{Bes-CS}. 
\begin{thm}  \label{cobhamsemenov}
 Let $\alpha,\beta$ be quadratic irrational numbers such that $\Q(\alpha)\ne\Q(\beta)$. Then every subset of $\N^n$ that is definable in both $\mc{S}_{\alpha}$ and $\mc{S}_\beta$, is definable in $(\N,+)$.
\end{thm}

\noindent Although we do not see how to obtain our result as a corollary, the proof of \cite[Theorem 3.1]{Bes-CS} can be adjusted straightforwardly once we prove the following analogue of \cite[Proposition 2.4]{Bes-CS}.
 
\begin{prop}\label{pumping}
Let $\alpha$ be a quadratic irrational number, and let $u,v,w\in\set{0,\dots,\floor\alpha}^*$ be such that $\abs v\ne0$, and either $[u]_{\mc S_\alpha}\ne0$ or $[v]_{\mc S_\alpha}\ne0$. Then there is $C>0$ such that 
  \[
  [uv^{P(\alpha) n}w]_{\mc S_\alpha}=(C+o(1))\norm {\Gamma_\alpha}^{\abs vn}.
  \]
\end{prop}
\begin{proof}
  We first argue that we may assume $w$ is the empty word. Indeed, first replace $w$ by the prefix of $v^{\abs w}$ of length $\abs w$. This will only change $[uv^{P(\alpha)n}w]_{\mc S_\alpha}$ by an additive constant. After this we may replace $v$ by the last $\abs v$ digits of $vw$, potentially enlarging $u$.\newline

  \noindent By \cref{shifting} there are non-zero constants $C',C''\in \R$ such that 
  \[[u0^{P(\alpha) n}]_{\mc S_\alpha}=(C'+o(1))\norm {\Gamma_\alpha}^n \text{ and } [v0^{P(\alpha)n}]_{\mc S_\alpha}=(C''+o(1))\norm {\Gamma_\alpha}^n.\]
   Since $[u]_{\mc S_\alpha}\ne0$ or $[v]_{\mc S_\alpha}\ne0$, either $C'$ or $C''$ is positive. Set 
   \[
   C=C'+C''((1-\norm {\Gamma_\alpha})^{-1}-1).
   \]
   Let $\varepsilon \in (0,1)$. By \cref{charpoly} we know that $\norm {\Gamma_\alpha}>1$. Hence the geometric series for $\norm {\Gamma_\alpha}^{-\abs v}$ converges. Let $M\in \N$ be such that for all $m\in\N_{\ge M}$ 
   \[
   \abs{(1-\norm {\Gamma_\alpha})^{-\abs v}-\sum_{i=0}^M\norm {\Gamma_\alpha}^{-i\abs v}}<\eps \text{ and }\abs{[v0^{P(\alpha)m\abs v}]_{\mc S_\alpha}-C''\norm {\Gamma_\alpha}^{m\abs v}}<\norm {\Gamma_\alpha}^{m\abs v}
   \]
Let $N\in \N_{\ge M}$ be such that for all $n\in \N_{\ge N-M}$
\[
\abs{u0^{P(\alpha)n}-C'\norm {\Gamma_\alpha}^n}<\eps\norm {\Gamma_\alpha}^n \text{ and  }\abs{[v0^{P(\alpha)n}]_{\mc S_\alpha}-C''\norm {\Gamma_\alpha}^n}<\eps M^{-1}\norm {\Gamma_\alpha}^{\abs vn}.\]
Finally, let $Q\in \N_{\ge N}$ be such that 
\[ 
[v^{P(\alpha)N}]_{\mc S_\alpha}<\eps\norm {\Gamma_\alpha}^{\abs vQ}.
\]
Now for all $n\in \N_{\ge Q}$
  \begin{align*}
    \abs{[uv^{P(\alpha)n}]_{\mc S_\alpha}-C\norm {\Gamma_\alpha}^{\abs vn}}=&\Big|[u0^{\abs vP(\alpha)n}]_{\mc S_\alpha}+\sum_{i=1}^M[v0^{P(\alpha)(n-i)\abs v}]_{\mc S_\alpha}\\
    &\quad+\sum_{i=M+1}^{n-N}[v0^{P(\alpha)(n-i)\abs v}]_{\mc S_\alpha}+[v^{P(\alpha)N}]_{\mc S_\alpha}\\
    &\quad-C'\norm{\Gamma_\alpha}^{\abs vn}-C''((1-\norm {\Gamma_\alpha})^{-1}-1)\norm{\Gamma_\alpha}^{\abs vn}\Big|\\
    \le&\abs{[u0^{\abs vP(\alpha)n}]_{\mc S_\alpha}-C'\norm {\Gamma_\alpha}^{\abs vn}}\\
    &\quad+\abs{\sum_{i=1}^MC''\norm {\Gamma_\alpha}^{(n-i)\abs v}-C''((1-\norm {\Gamma_\alpha})^{-\abs v}-1)\norm {\Gamma_\alpha}^{\abs vn}}\\
    &\quad+\sum_{i=1}^M\abs{[v0^{P(\alpha)(n-i)\abs v}]_{\mc S_\alpha}-\norm {\Gamma_\alpha}^{(n-i)\abs v}}\\
    &\quad+\sum_{i=M+1}^{n-N}[v0^{P(\alpha)(n-i)\abs v}]_{\mc S_\alpha}+[v^{P(\alpha)N}]_{\mc S_\alpha}\\
    <&\eps\norm {\Gamma_\alpha}^{\abs vn}+C''\eps\norm {\Gamma_\alpha}^{\abs vn}+\sum_{i=1}^M\abs{[v0^{P(\alpha)(n-i)\abs v}]_{\mc S_\alpha}-\norm {\Gamma_\alpha}^{(n-i)\abs v}}\\
    &\quad +\sum_{i=M+1}^{n-N}[v0^{P(\alpha)(n-i)\abs v}]_{\mc S_\alpha}+\eps\norm {\Gamma_\alpha}^{\abs vn}\\
    \le&(C''+2)\eps\norm {\Gamma_\alpha}^{\abs vn}+\sum_{i=1}^M\eps M^{-1}\norm {\Gamma_\alpha}^{\abs vn}+\sum_{i=M+1}^{n-N}(C''+1)\norm {\Gamma_\alpha}^{(n-i)\abs v}\\
    \le&(C''+3)\eps\norm {\Gamma_\alpha}^{\abs vn}+(C''+1)\norm {\Gamma_\alpha}^{\abs vn}\sum_{i=M+1}^\infty\norm {\Gamma_\alpha}^{-\abs vi}\\
    <&(2C''+4)\eps\norm {\Gamma_\alpha}^{\abs vn}.
  \end{align*}
\end{proof}

\begin{proof}[Proof of \cref{cobhamsemenov}]
 We use the proof of \cite{Bes-CS}, using our \cref{pumping} instead of their Proposition 2.4. Mainly, one just needs to replace the occurrences of $uv^nw$ with $uv^{pn}w$, in particular in the definition of $X$ on page 209. We leave the details to the reader.
\end{proof}

\subsection{The general case} We are now ready to finish the proofs of Theorem A and B. We restate Theorem B using our notation.

\begin{thm}\label{thm:B}
Let $\alpha,\beta  \in \R_{>1}$ be multiplicatively independent irrational Pisot numbers such that $\Q(\alpha)\cap\Q(\beta)=\Q$, and let $X\subseteq[0,1]^d$ be both sequentially $\mathcal{S}_{\alpha}$- and $\mathcal{S}_{\beta}$-recognizable. Then $X$ is definable in $(\R,{<},+,\Z)$.
\end{thm}
\begin{proof}
Let $X\subseteq [0,1]^n$ be 
both sequentially $\mathcal{S}_{\alpha}$- and $\mathcal{S}_{\beta}$-recognizable.
Towards a contradiction, suppose that $X$ is not definable in $(\R,{<},+,\Z)$.
By \cref{definmultidim}, we can assume that $n=1$ and $X$ is both $\alpha$- and $\beta$-product stable in $[0,1]$.
  Set 
\[
Z:=\{\alpha^kx \ : \ k\in\N, x\in X\}.
\]
Obviously, $Z$ is $\alpha$-product stable in $\R_{>0}$.
We observe that $Z$ is sequentially $\mathcal{S}_{\alpha}$-recognizable, because for all $u\in\set{0,\dots,\floor{\alpha}}^*,v\in\set{0,\dots,\floor{\alpha}}^\omega$ 
\[
[u\star v]_{\mc S_\alpha}\in Z \text{ if and only if } [\star uv]_{\mc S_\alpha}\in X.
\]
Clearly, $Z$ is $\alpha$-product stable in $\R_{>0}$. Since $X$ is both $\alpha$- and $\beta$-product stable in $[0,1]$, we get that
\[
Z=\{\beta^k x\ : \ k\in\N,x\in X\}.
\]
Thus $Z$ is $\beta$-product stable in $\R_{>0}$. By \cref{prodstabletriv} we know that $Z$ is either $\emptyset$ or $\R_{>0}$. However, it follows from $\alpha$-product stability that  $Z\cap[0,1]=X$. This contradicts that $X$ is not definable in $(\R,{<},+,\Z)$.
\end{proof}

\begin{proof}[Proof of Theorem A]
By \cref{definbounded} we may assume that either $X\subseteq\N^n$ or $X\subseteq[0,1]^n$. If $X\subseteq\N^n$, then \cref{cobhamsemenov} applies. Suppose that $X\subseteq[0,1]^n$. Then $X$ is both sequentially $\alpha$- and sequentially $\beta$-regular by \cref{thm:recdef} and we can apply \cref{thm:B}.
\end{proof}

\section{Conclusion}

In this paper, we prove a Cobham-Sem\"enov theorem for scalar multiplication: let $\alpha,\beta \in \R_{>0}$ be such that $\alpha,\beta$ are quadratic and $\Q(\alpha)\neq \Q(\beta)$, then every set definable in both $\mathcal{R}_{\alpha}$ and $\mathcal{R}_{\beta}$, is already definable in $(\R,{<},+,\Z)$. None of the assumptions on $\alpha$ and $\beta$ can be dropped. It is clear that if $\Q(\alpha)=\Q(\beta)$ and $\alpha$ is irrational, then both $\mc R_{\alpha}$ and $\mc R_{\beta}$ define multiplication by $\alpha$, yet this function is not definable in $(\R,{<},+,\Z)$. Furthermore, suppose that $\alpha$ is not quadratic. As noted in the introduction,  in this situation $\mathcal{R}_{\alpha}$ defines every arithmetical subset of $\N^m$, since $1,\alpha,\alpha^2$ are $\Q$-linearly independent. If $\beta$ is also not quadratic, all arithmetical subsets satisfy the assumption, but not all of them the conclusion of Theorem A. Even when $\beta$ is quadratic, one can check that the set of denominators of the convergents of $\beta$ is definable in both $\mc R_{\alpha}$ and $\mc{R}_{\beta}$, assuming $\alpha$ is not quadratic. This set is not definable in $(\R,{<},+,\Z)$ when $\beta$ is irrational, witnessing the failure of the conclusion of Theorem A.\newline

\noindent As part of the proof of Theorem A, we establish in Theorem B a similar result for bounded subset of $\R^n$ that are $\alpha$- and $\beta$-recognizable in the sense of \cite{CLR}, where  $\alpha,\beta  \in \R_{>1}$ are multiplicatively independent irrational Pisot numbers such that $\Q(\alpha)\cap\Q(\beta)=\Q$. Following the argument in \cite[p.118]{CLR}, this theorem can be used to extend Adamczeski and Bell's Cobham-style theorem for fractals \cite[Theorem 1.4]{AdamczeskiBell} and its higher dimensional analogues as proven in \cite{CLR} and Chan and Hare \cite{CH-Semenov} to $\alpha$- and $\beta$-self-similar sets as defined in \cite[Definition 60]{CLR}.\newline

\noindent There are several immediate questions we have to leave open. In particular, we do not know whether Theorem A holds if we replace definability by definability with parameters, and whether Theorem B
holds for unbounded subsets of $\R^n$.
\bibliographystyle{amsplain}
\bibliography{biblio}
\end{document}